\documentclass[12pt]{article}

\usepackage{mathtools}
\usepackage{amssymb}
\usepackage{amsthm}
\usepackage{tikz}
\usepackage{tikz-3dplot}
\usetikzlibrary{calc}
\usepackage{pgfplots}
\usepackage[
font={footnotesize}]{caption}
\pgfplotsset{compat=newest}
\usepackage{fullpage}
\usepackage{url}

\DeclareSymbolFont{bbold}{U}{bbold}{m}{n}
\DeclareSymbolFontAlphabet{\mathbbold}{bbold}

\newcommand{\N}{\mathbb{N}}

\newcommand{\R}{\mathbb{R}}
\newcommand{\C}{\mathbb{C}}
\newcommand{\K}{\mathbb{K}}

\newcommand{\id}{\operatorname{id}}

\newcommand{\dd}{\mathrm{d}}

\newcommand{\cH}{\mathcal{H}}
\newcommand{\cA}{\mathcal{A}}
\newcommand{\cK}{\mathcal{K}}
\newcommand{\cL}{\mathcal{L}}
\newcommand{\cM}{\mathcal{M}}

\newcommand{\cB}{\mathcal{B}}
\newcommand{\cV}{\mathcal{V}}

\renewcommand{\epsilon}{\varepsilon}

\DeclareMathAccent{\Circ}{\mathalpha}{operators}{"17}
\newcommand{\interior}[1]{\operatorname{\Circ{#1}}}



\DeclareMathOperator{\spt}{spt}

\DeclareMathOperator{\dive}{div}
\DeclareMathOperator{\curl}{curl}
\DeclareMathOperator{\curlr}{curl_{\textnormal{red}}}
\DeclareMathOperator{\curlrd}{curl_{\textnormal{red}}^\diamond}
\DeclareMathOperator{\curld}{{curl}^\diamond}

\newcommand{\curlc}{\operatorname{\mathring{curl}}}
\DeclareMathOperator{\grad}{grad}
\newcommand{\gradc}{\operatorname{\mathring{grad}}}
\newcommand{\gradcd}{\operatorname{\mathring{grad}^\diamond}}

\renewcommand{\ker}{\operatorname{ker}}

\DeclareMathOperator{\ran}{ran}
\DeclareMathOperator{\dom}{dom}

\renewcommand{\Re}{\operatorname{Re}}
\DeclareMathOperator{\lin}{lin}

\newcommand{\rdd}{{\rm red}}

\let\phi\varphi

\let\leq\leqslant

\let\geq\geqslant

\makeatletter
\def\@row#1,{#1\@ifnextchar;{\@gobble}{&\@row}}
\def\@matrix{%
    \expandafter\@row\my@arg,;%
    \@ifnextchar({\\ \get@in@paren{\@matrix}}{\after@matrix}%
    }
\def\matrixtype#1#2#3{%
    \ifmmode\def\after@matrix{\end{#2}\right#3}%
    \else\def\after@matrix{\end{#2}\right#3$}$\fi
    \left#1\begin{#2}\get@in@paren{\@matrix}%
    }
\def\@column#1,{#1\@ifnextchar;{\@gobble}{\\ \@column}}
\newcommand\vect{}
\def\svect(#1){\left(\begin{smallmatrix}\@column#1,;\end{smallmatrix}\right)}
\def\vect{\get@in@paren{\@vect}}
\def\@vect{\left(\begin{matrix}\expandafter\@column\my@arg,;\end{matrix}\right)}
\def\get@in@paren#1({\def\my@arg{}\def\my@rest{}\def\after@get{#1}\get@arg}
\let\e@a\expandafter
\def\get@arg#1){\e@a\kl@test\my@rest#1(;}
\def\kl@test#1(#2;{\e@a\def\e@a\my@arg\e@a{\my@arg#1}%
                   \ifx:#2:\let\my@exec\after@get
                   \else\let\my@exec\get@arg
                        \e@a\def\e@a\my@arg\e@a{\my@arg(}%
                        \def@rest#2;%
                   \fi\my@exec}
\def\def@rest#1(;{\def\my@rest{#1\kl@zu}}
\def\kl@zu{)}

\makeatletter
\newcommand\MyPairedDelimiter{%
  \@ifstar{\My@Paired@Delimiter{{}}}
          {\My@Paired@Delimiter{}}%
}
\newcommand\My@Paired@Delimiter[4]{%
  \newcommand#2{%
    \@ifstar{\start@PD{#1}{\delimitershortfall=-1sp}{#3}{#4}}
            {\start@PD{#1}{}{#3}{#4}}%
  }%
}
\newcommand\start@PD[5]{%
  #1\mathopen{\mathpalette\put@delim@helper{\put@delim{#2}{#3}{.}{#5}}}%
  #5%
  \mathclose{\mathpalette\put@delim@helper{\put@delim{#2}{.}{#4}{#5}}}%
}
\newcommand\put@delim@helper[2]{%
  \hbox{$\m@th\nulldelimiterspace=0pt #2#1$}%
}
\newcommand\put@delim[5]{%
  \setbox\z@\hbox{$\m@th#5{#4}$}%
  \setbox\tw@\null
  \ht\tw@\ht\z@ \dp\tw@\dp\z@
  #1#5%
  \left#2\box\tw@\right#3%
}

\makeatother
\MyPairedDelimiter*{\abs}{\lvert}{\rvert}
\MyPairedDelimiter*{\norm}{\lVert}{\rVert}
\MyPairedDelimiter{\set}{\{}{\}}

\newcounter{intthm}

\theoremstyle{plain} 
\newtheorem{theorem}{Theorem}[section]
\newtheorem{thmABC}{Theorem}
\newtheorem{corollary}[theorem]{Corollary}
\newtheorem{lemma}[theorem]{Lemma}
\newtheorem{proposition}[theorem]{Proposition}
\theoremstyle{definition}
\newtheorem{example}[theorem]{Example}

\newtheorem{problem}[theorem]{Problem}

\newtheorem{remark}[theorem]{Remark}

\usepackage{enumitem}

\setenumerate[1]{nolistsep} 
\setenumerate[2]{nolistsep} 

\allowdisplaybreaks
\numberwithin{equation}{section}

\begin{document}

\title{Nonlocal $H$-convergence for topologically nontrivial domains}

\author{Marcus Waurick\footnote{Institute for Applied Analysis,
Faculty for Mathematics and Computer Science, 
TU Bergakademie Freiberg,
Freiberg, Germany, Email:\ {\tt marcus.wau\rlap{\textcolor{white}{hugo@egon}}rick@math.\rlap{\textcolor{white}{balder}}tu-freiberg.de}}}

\date{}

\maketitle

\begin{abstract} 
The notion of nonlocal $H$-convergence is extended to domains with nontrivial topology, that is, domains with non-vanishing harmonic Dirichlet and/or Neumann fields. If the space of harmonic Dirichlet (or Neumann) fields is infinite-dimensional, there is an abundance of choice of pairwise incomparable  topologies generalising the one for topologically trivial $\Omega$. It will be demonstrated that if the domain satisfies the Maxwell compactness property the corresponding natural version of the corresponding (generalised) nonlocal $H$-convergence topology has no such ambiguity. Moreover, on multiplication operators the nonlocal $H$-topology coincides with the one induced by (local) $H$-convergence introduced by Murat and Tartar. The topology is used to obtain nonlocal homogenisation results including convergence of the associated energy for electrostatics. The derived techniques prove useful to deduce a new compactness criterion relevant for nonlinear static Maxwell problems. \end{abstract}



\medmuskip=4mu plus 2mu minus 3mu
\thickmuskip=5mu plus 3mu minus 1mu
\belowdisplayshortskip=9pt plus 3pt minus 5pt






\newpage
\tableofcontents

\newpage

\section{Introduction}\label{s:int}

Many physical phenomena such as the classical wave equation, Maxwell's equations or the heat equation consist of coefficients describing the specific properties of the material the effects are considered in. These specifics are also provided in so-called constitutive relations. Classically, these relations are represented by (matrix-valued) functions depending on the material point. With this it is possible to model a composite consisting of two materials with, say, different thermal conductivities. In order to model more sophisticated physical phenomena it is not enough to restrict oneself to matrix-valued functions modelling the material's properties. In fact, also nonlocal effects need to be taken into account. This is particularly the case for the description of interacting particles, where the McKean--Vlasov equations are used, see, e.g., \cite{Carillo18}. We refer to the references in \cite{Carillo18} for a multitude of applications and to \cite[Example 2.7]{W18_NHC} for a linear version of these equations. 

The modelling of the electro-magnetic effects of metamaterials characterised by nonlocal constitutive relations is provided for instance in \cite{Gorlach2016,Ciattoni2015,Mendez2017}. Moreover, a bespoke homogenisation process has been provided in these references to have a workable model at hand. The type of equations considered are of the form
\[
    \curl (1+\rho_n*)\curl E_n = f\in L_2(\Omega)^3 
\]
\textcolor{black}{ in three spatial dimensions} with $E_n\in L_2(\Omega)^3 $ such that $ \curl E_n\in L_2(\Omega)^3$ subject to \textcolor{black}{ homogeneous boundary conditions for the tangential trace of $E_n$ (also known as perfect electric conductor boundary condition, we shall use `electric boundary condition' as a short-hand, see \cite[eq.~(1.6)]{LS19})}, $(\rho_n)_n$ a sequence of integral kernels defined on $\Omega\times \Omega$ and \textcolor{black}{all $\curl$ operations are taken in the distributional sense}.

The sources \cite{Gorlach2016,Ciattoni2015,Mendez2017} were the reason for the invention of the notion of \emph{nonlocal $H$-convergence} in \cite{W18_NHC}. This notion describes the functional analytic properties the homogenisation processes both local and nonlocal problems have in common. In order to briefly describe the notion of both nonlocal and local $H$-convergence, we let $\Omega\subseteq\R^3$ be open and bounded with connected complement. Then, \textcolor{black}{ given $0<\alpha<\beta$,} consider a sequence $(\varepsilon_n)_n$ in \textcolor{black}{\begin{equation}\label{eq:Mab}
M(\alpha,\beta;\Omega)\coloneqq \{ a\in \cM(\alpha,\beta); a\in L_\infty(\Omega)^{3 \times 3}\}.
\end{equation}}This sequence is said to \textbf{(locally) $H$-converge} to $\varepsilon\in M(\alpha,\beta;\Omega)$, if for all $f\in H^{-1}(\Omega)(=H_0^1(\Omega)')$ the following properties are satisfied. Consider the sequence of solutions $(u_n)_n$ in $H_0^1(\Omega)$ with
\[
   -\dive \varepsilon_n(x)\grad u_n = f
\]in the distributional sense. Then $(u_n)_n$ weakly converges to $u\in H_0^1(\Omega)$ and $(\varepsilon_n\grad u_n)_n$ weakly converges to $\varepsilon\grad u\in L_2(\Omega)^3$, where $u$ satisfies
\[
   -\dive \varepsilon(x)\grad u = f.
\]$M(\alpha,\beta;\Omega)$ can now be endowed with the topology induced by (local) $H$-convergence. This set becomes a metrisable compact Hausdorff space, see \cite{Murat1997,Tartar2009}. It can be shown that for the corresponding second order $\curl$-problem a similar property holds, see also \cite{Nicaise2020}. For this, let $\curl$ be the distributional $\curl$ derivative realised as a closed, densely defined operator in $L_2(\Omega)^3$ with maximal domain, that is, \textcolor{black}{
\begin{equation}\label{eq:curlinto}
   \curl \colon \dom(\curl)\subseteq L_2(\Omega)^3\to L_2(\Omega)^3,\phi\mapsto \nabla\times \phi,
\end{equation}where $\nabla\times$ acts in the distributional sense and 
\[
   \dom(\curl) = \{ \phi\in L_2(\Omega)^3; \nabla\times \phi \in  L_2(\Omega)^3\}.
\] It is not difficult to show that $\curl$ is closed and densely defined. Assuming $\ran(\curl)\subseteq L_2(\Omega)^3$ to be closed and using $ \ker(\curl)^\bot$, the orthogonal complement of $\ker(\curl)$ in $L_2(\Omega)^3$,  we define the \emph{reduced $\curl$-operator}, $\curlr$, by
\[
  \curlr \colon \dom(\curl)\cap \ker(\curl)^\bot \subseteq  \ker(\curl)^\bot \to \ran(\curl), \phi\mapsto \curl \phi. 
\] In contrast to $\curl$, $\curlr$ is one-to-one and onto.} We denote by $H^1(\curlr)$ the domain of $\curlr$ endowed with the graph norm. Now, if $(\varepsilon_n)_n$ locally $H$-converges to some $\varepsilon$, then for all $g\in H^{-1}(\curlr)\coloneqq H^1(\curlr)'$ and (uniquely determined) $v_n\in H^1(\curlr)$ satisfying
\[
   \curlrd \varepsilon_n(x)^{-1}\curlr v_n =g,
\]where $\curlrd$ is \textcolor{black}{the dual of $\curlr$ being interpreted as a bounded linear operator, see \eqref{eq:Cdi} below\footnote{In the present situation, we interpret
\[
  \curlr \colon H^1(\curlr) \to \ran(\curl), \phi\mapsto \curl \phi
\]
and obtain
\[
  \curlrd \colon \ran(\curl) \to H^{-1}(\curlr), \psi\mapsto (H^1(\curlr)\ni \phi\mapsto \langle \psi, \curl \phi\rangle).
\]
Note that this is paralleled by the similar construction for $\grad$ on $H_0^1(\Omega)$ and $\dive$ with target space $H^{-1}(\Omega)$}}, we obtain $v_n \to v$ weakly and $\varepsilon_n^{-1}\curlr v_n \to \varepsilon^{-1}\curlr v$ weakly in the respective spaces $H^1(\curlr)$ and $L_2(\Omega)^3$ with $v\in H^1(\curlr)$ uniquely solving
\[
     \curlrd \varepsilon(x)^{-1}\curlr v =g.
\]

In \cite{W18_NHC} this consequence of local $H$-convergence is used in order to generalise the class of admissible coefficients to also include coefficients of convolution type mentioned above. Indeed, basing on the fact that for $\Omega$ with connected complement the space of harmonic Dirichlet fields, that is, 
\[
\cH_D(\Omega)\coloneqq\ker(\dive)\cap \ker(\curl^*)
\] is $0$-dimensional, we have the decomposition
\begin{equation}\label{eq:HDint}
   L_2(\Omega)^3 = \ran(\gradc)\oplus \ran(\curl).
\end{equation}
\textcolor{black}{Here,
\begin{align*}
  \gradc &\colon H_0^1(\Omega) \subseteq L_2(\Omega)\to L_2(\Omega)^3, u\mapsto \nabla u, \\
  \dive & \colon \dom(\dive)\subseteq L_2(\Omega)^3 \to L_2(\Omega), \phi \mapsto \nabla\cdot \phi,
\end{align*}
with $\dom(\dive)=\{\phi \in L_2(\Omega)^3; \nabla\cdot \phi\in L_2(\Omega)^3\text{ in the distributional sense}\}$ and $\curl^*$ is the $L_2(\Omega)^3$-adjoint from $\curl$ given in \eqref{eq:curlinto}.}
Thus, by \eqref{eq:HDint}, any bounded linear operator $\varepsilon\in \cB(L_2(\Omega)^3)$ admits a representation as $2$-by-$2$ matrix according to this decomposition. Then, in the present situation of vanishing harmonic Dirichlet fields, a \textcolor{black}{bounded and equi-coercive, see \eqref{eq:sMabint} below,} sequence of bounded linear operators $(\varepsilon_n)_n$ is said to converge to $\varepsilon$ \textbf{in the nonlocal $H$-sense}, if \begin{align*}
 &\begin{cases}
  \forall f\in H^{-1}(\Omega)=H_0^1(\Omega)' &\\
  \forall u_n\in H_0^1(\Omega) &\\
  -\dive \varepsilon_n \grad u_n = f&
\end{cases} \Longrightarrow \begin{cases}u_n\to u \in H_0^1(\Omega)\text{ weakly}& \\ \varepsilon_n \grad u_n\to \varepsilon \grad u \in L_2(\Omega)^3 \text{ weakly}\end{cases} \text{ and } \\
& \begin{cases}
  \forall g\in H^{-1}(\curlr) &\\
  \forall v_n\in H^1(\curlr) &\\
  \curlrd \varepsilon_n^{-1} \curl v_n = g&
\end{cases}\Longrightarrow \begin{cases} v_n\to v \in H^1(\curlr)\text{ weakly}&\\ \varepsilon_n^{-1} \curl v_n\to \varepsilon^{-1} \curl v \in L_2(\Omega)^3\text{ weakly}&\end{cases}
\end{align*}

The introduction of the concept of nonlocal $H$-convergence leads then to a well-defined operator topology, $\tau_{\textrm{nlH}}$, on (a subset of) $\cB(L_2(\Omega)^3)$, which is weaker than both strong operator and norm topology but incomparable to the weak operator topology. It can be shown that $\tau_{\textrm{nlH}}$ is \emph{not} a linear topology. In any case, the developed concept hinges on the vanishing of harmonic Dirichlet fields. It is the aim of the present article to understand the situation of when the space of harmonic Dirichlet fields is non-zero. Domains with existing non-zero harmonic Dirichlet fields or harmonic Neumann fields, $\cH_N(\Omega) \coloneqq \ker(\gradc)\cap\ker(\curl)$, are called \textbf{topologically nontrivial}. \textcolor{black}{In \cite{Picard1982,Pauly2021} for continuous, (weak) Lipschitz domains, the dimension of $\cH_N(\Omega)$ and $\cH_D(\Omega)$ has been computed in terms of the topology of $\Omega$. More precisely, for a bounded $\Omega$, the number of bounded components of the complement is the dimension of $\cH_D(\Omega)$, the number of `handles' is the dimension of $\cH_N(\Omega)$, see \cite[Theorem 1]{Picard1982} or \cite[Section 2]{Pauly2021} and Figure \ref{fig:hDN} for examples.}
\begin{figure}[t]
\begin{minipage}{0.32\textwidth}
\centering
\tdplotsetmaincoords{0}{0} 
\begin{tikzpicture}[tdplot_main_coords]
 \shade[ball color=black!20, opacity=1] (0,0,0) circle (1);
\end{tikzpicture}
\end{minipage}
\begin{minipage}{0.32\textwidth}
\centering
  \begin{tikzpicture}[yscale=cos(70)]
    \draw[double distance=5mm] (0:1) arc (0:180:1);
    \draw[double distance=5mm] (180:1) arc (180:360:1);
  \end{tikzpicture}
  \end{minipage}
\begin{minipage}{0.32\textwidth}
\centering
 \tdplotsetmaincoords{0}{0} 
\begin{tikzpicture}[tdplot_main_coords]
 \shade[ball color=black!20, opacity=1] (0,0,0) circle (1);
  \clip (0,0) circle (0.6);
  \shade[ball color=white] (0,0,0) circle (0.6);
\end{tikzpicture}
\end{minipage}
\caption{(From left to right) The ball is topologically trivial. 
The torus has $1$-dimensional space of harmonic Neumann fields with $\cH_D(\Omega)=\{0\}$.  For the annular region of a ball taken out an inner ball, $\cH_D(\Omega)$ is one-dimensional and $\cH_N(\Omega)=\{0\}$.}\label{fig:hDN}
\end{figure}

\textcolor{black}{Turning to the definition of nonlocal $H$-convergence, we note that `nonlocal' is related to `local' the same way as `nonlinear' is related to `linear'. We define, for a bounded linear operator $T$ acting on some Hilbert space, $\Re T\coloneqq \frac12(T+T^*)$. Defining the nonlocal analogue of $M(\alpha,\beta;\Omega)$,} for $0<\alpha<\beta$, we introduce
\begin{equation}\label{eq:sMabint}
\cM(\alpha,\beta)\coloneqq \{a\in \cB(L_2(\Omega)^3); \Re a\geq \alpha, \Re a^{-1}\geq 1/\beta\}.
\end{equation} We say that $\tau$ is a \textbf{nonlocal $H$-topology}\footnote{We provide a slightly more general setting in Section \ref{sec:nHc} below.} on $\cM(\alpha,\beta)$, if there exists a closed subspace $\cH_0\subseteq L_2(\Omega)^3$ such that $\ran(\gradc)\subseteq \cH_0$ and $\ran(\curl)\subseteq \cH_0^\bot\eqqcolon \cH_1$ such that identifying any operator $a\in \cM(\alpha,\beta)$ with $(a_{jk})_{j,k\in \{0,1\}}\in \cB(\cH_0\oplus \cH_1)$ the topology $\tau$ is the initial topology induced by
\begin{equation}\label{eq:indnon}
   a\mapsto a_{00}^{-1},\ a\mapsto a_{00}^{-1}a_{01},\ a\mapsto a_{10}a_{00}^{-1},\ a\mapsto a_{11}-a_{10}a_{00}^{-1}a_{01}
\end{equation} with the target spaces all endowed with corresponding weak operator topologies. \textcolor{black}{In the topological trivial setting, there is only one choice for $\cH_0$ and $\cH_1$ here. Then, it has been shown in \cite{W18_NHC} that this topology  coincides with local $H$-convergence on $M(\alpha,\beta;\Omega)$. In order to get the general idea why this definition makes sense and somehow captures homogenisation picture the one-dimensional situation and consider the classical homogenisation problem for periodic coefficients. Then the harmonic mean is the effective coefficient and is obtained by the $L_\infty$-weak-$*$-limit of the inverses of the coefficients. Expressed in operator-theoretic terms, this corresponds to the convergence of the inverses in the weak operator topology of the corresponding $L_2$-space the problem is considered in. In the present three-dimensional situation a similar effect occurs as with stratified media, where one needs to take into account Schur-complements of some of the components of the coefficient matrix, see, e.g., \cite[Theorem 5.12]{Cioranescu1999}. Hence, the more involved topology induced by the mappings \eqref{eq:indnon} that needs to be considered here. Note that the detailed definition is provided in Section \ref{sec:nHc}.}

The first major result of this article reads as follows (part (a) turns out to be rather trivial, whereas part (b) is the main part of the theorem). \textcolor{black}{In order to anticipate potential future developments in mathematical research, we use the implicit assumption on $\Omega$ to be such that the ranges of $\gradc$ and $\curl$ are closed. For this note that the assumption on $\Omega$ is implicit. We note that $\Omega$ bounded with Lipschitz boundary is sufficient, see Example \ref{ex:rangeclosed} for the details.}
\begin{thmABC}[consistence and uniqueness of nonlocal $H$-topologies]\label{thm:exunitop} (a) For all $\Omega\subseteq \R^3$ open with $\ran(\gradc),\ran(\curl)\subseteq L_2(\Omega)^3$ closed, define 
\[
\mathfrak{T} \coloneqq \{ \tau; \tau \text{ nonlocal $H$-topology on }\cM(\alpha,\beta)\}.
\]
Given $(a_n)_n$, $a$ in $\cM(\alpha,\beta)$ and if $\cH_D(\Omega)=\{0\}$, then for all $\tau_\Omega\in \mathfrak{T}$
\[
    a_n \to a \text{ in }\tau_\Omega \iff     a_n \to a \text{ in }\tau_{\textnormal{nlH}}.
\]
(b) If $\dim(\cH_D(\Omega))<\infty$, $\mathfrak{T}$ is a singleton.
\end{thmABC}

We provide three applications of the Theorem \ref{thm:exunitop} and the topology constructed, next. For these applications, we use the following standing assumption. We recall that there is no assumption included on the topology of $\Omega$ particularly allowing for the cases exemplified in Figure \ref{fig:hDN}.
\begin{center}
 {\it Let $\Omega\subseteq \R^3$ open, bounded with Lipschitz boundary and $\tau_\Omega$ be the unique topology on $\cM(\alpha,\beta)$ as in (b) of Theorem \ref{thm:exunitop}.}
\end{center}
The first application concerns the homogenisation problem associated to electrostatics. The main challenge is to understand the relationship of the coefficients to the projection of the modified harmonic Dirichlet field
\[
\cH_{\varepsilon_n}\coloneqq\ker(\dive\varepsilon_n)\cap \ker(\curl^*)
\] In the course of this manuscript, we will define a (not necessarily orthogonal) continuous projection
\[
   \pi_{\varepsilon} \colon L_2(\Omega)^3 \to \cH_{\varepsilon}.
\] The corresponding projection for $\varepsilon=\id$ is denoted by $\pi_D$ (orthogonally) projecting onto $\cH_D(\Omega)$. We shall establish the following result \textcolor{black}{on the homogenisation for electrostatic equations. We refer to \cite[12-1]{FL11} for the equations of electrostatics; to \cite[eq.~(1.6)]{LS19} for the corresponding boundary condition for a perfect electric conductor, which is conveniently addressed by the `electric boundary condition' as a short hand in the following. Note that the involvement of the operator $\curlrd$ is a possible functional analytic implementation of this boundary condition; see also Remark \ref{rem:rdc} below. As it has been shown in \cite{Picard1982}, in order to warrant well-posedness for the equations in questions one needs to control also the space of harmonic Dirichlet vector-fields; therefore, it is required to include the operators $\pi_D$ and $\pi_\varepsilon$. The expression $\langle \varepsilon E,E\rangle$ will be referred to as the `energy' or `stored energy' for electrostatics (it coincides with the actual stored energy in this situation up to a factor of $1/2$, see \cite[p 64]{E93}).}
\begin{thmABC}[nonlocal homogenisation theorem for electrostatics and convergence of energy]\label{thm:hometstatic} Let $(\varepsilon_n)_n$, $\varepsilon$ in $\cM(\alpha,\beta)$ with $\varepsilon_n\to \varepsilon$ in $\tau_\Omega$.  For all $f\in H^{-1}(\Omega)$, $g\in H^{-1}(\curlr)$ and $x \in \cH_{D}(\Omega)$ and $E_n \in L_2(\Omega)^3$ solving
\[
   \dive \varepsilon_n E_n =f,\ \curlrd E_n = g,\ \pi_D\pi_{\varepsilon_n} E_n = x,
\]we obtain
\[
  E_n \to E\text{ weakly in }L_2(\Omega)^3,
\]where $E\in L_2(\Omega)^3$ is the unique solution of
\[
   \dive \varepsilon E  =f,\ \curlrd E  = g,\ \pi_D\pi_\varepsilon E = x.
\]Moreover, $\varepsilon_nE_n\to \varepsilon E$ weakly and
\[
   \langle \varepsilon_n E_n,E_n\rangle_{L_2} \to    \langle \varepsilon E,E\rangle_{L_2}.
\]
\end{thmABC}
The actual result is more general and the solution theory for this kind of system of equations is based on the solution theory developed in \cite{TW14_FE}. Since in \cite{TW14_FE} nonlinear equations have been considered, possible nonlinear generalisations similar to the ones in \textcolor{black}{\cite{Picard1990}} may be provided in future. The main technique for the solution theory relies on a Helmholtz type decomposition previously known for self-adjoint and coercive coefficients. 

As a second application, we show that the introduced nonlocal topology coincides with the one of local $H$-convergence on the space of multiplication operators 
\[M(\alpha,\beta;\Omega)= \{ a\in \cM(\alpha,\beta); a\in L_\infty(\Omega)^{3 \times 3}\}.
\]
\begin{thmABC}[compatibility of nonlocal $H$-convergence with local $H$-convergence]\label{thm:locn} For $(a_n)_n$, $a$ in $M(\alpha,\beta;\Omega)$ the following conditions are equivalent:
\begin{enumerate}
  \item[(i)] $a_n\to a$ in the local $H$-sense;
  \item[(ii)] $a_n\to a$ in $\tau_\Omega$.
\end{enumerate}
\end{thmABC}

In the third application, we derive a new compactness result generalising the Picard--Weber--Weck compactness theorem in the variable coefficient case to `moving coefficients'. This is particularly relevant for nonlinear equations with dielectricity depending on the electric field. \textcolor{black}{Some equations with nonlinearity of a similar type have been considered mathematically in \cite{DTW24,Picard1990}. For references in the physics literature, we exemplarily refer to \cite{Sh89} and to \cite[Chapter 2]{G80}. For $\Omega$ topologically trivial, in order to prove for a given $f$ and $g$, existence of solutions of systems of the type
\begin{equation}
   \dive \varepsilon(E) E = f,\quad \curl E = g,
\end{equation} with $E$ subject to the electric boundary condition, one might want to apply Schauder's fixed point theorem. In this case, one needs to confirm that the mapping assigning to $G\in L_2(\Omega)^3$ the solution $\tilde{E}$ of
\[
\dive \varepsilon(G) {\tilde{E}} = f,\quad \curl \tilde{E} = g;
\] is compact. With only little regularity properties for $\varepsilon$, this compactness might be very difficult to show; this is where the following compactness theorem might be of some use.}
\begin{thmABC}[compactness criterion for moving coefficients]\label{thm:comprough} Let $(\varepsilon_n)_n$, $\varepsilon$ in $\cM(\alpha,\beta)$. 

Assume $\varepsilon_n\to \varepsilon$ in the weak operator topology and in $\tau_\Omega$; and for all $E\in L_2(\Omega)^3$, assume $\{\dive \varepsilon_n E; n\in \N\}$ to be relatively compact in $H^{-1}(\Omega)$. 

Let $(E_n)_n$ be a  bounded sequence in $\dom(\curl^*)$ and assume $\{\dive \varepsilon_n E_n; n\in \N\}$ to be bounded in $L_2(\Omega)$. Then $(E_n)_n$ contains an $L_2(\Omega)^3$-convergent subsequence.
\end{thmABC}

The methods developed here also work for a different set of boundary conditions. Thus, analogous results also hold for the magneto-static equations; \textcolor{black}{the corresponding boundary conditions are the so-called `Transverse Magnetic Boundary Conditions' or `magnetic boundary condition' in short, see, e.g., \cite[eq.~(1.59)]{LS19}.} In particular, a compactness statement with moving magnetic permittivities also holds even though, we will not explicitly spell out the result here. The whole theory developed here -- similar to the notion of nonlocal $H$-convergence in \cite{W18_NHC} -- is based on the concept of (closed/compact) Hilbert complexes. In order to keep the present manuscript as accessible as possible we refrain from providing the general statements in the framework of Hilbert complexes and rather refer to future research. Consequently, we shall also address the particular cases for other established closed/compact Hilbert complexes in a forthcoming publication. 

Before we provide a brief overview of the arrangement of the contents of the paper we provide a short literature review to relevant results regarding the article at hand. 

The idea of addressing possibly non-periodic homogenisation problems by introducing an abstract notion of convergence goes back to Spagnolo (for symmetric coefficients) in \cite{Spagnolo1967,Spagnolo1976} and, independently, to Murat and Tartar (see, e.g.,~\cite{Murat1997,Tartar2009}) (for general matrix-valued multiplication operators). The generalisation of these concepts to homogenisation for nonlocal problems was introduced in \cite{W18_NHC}. It has been applied to variational problems in the situation of curl-div problems in \cite{Nicaise2020}. In this source an abstract notion for nonlocal $H$-convergence has been introduced, which is similar to the one to be introduced below. However, the concepts have been applied to topologically trivial domains only (or even more restrictive situations). We note that for topologically trivial domains, the notion of nonlocal $H$-convergence has been used to deduce (nonlocal) homogenisation theorems for time-dependent Maxwell's equations or the wave equation, see \cite{W18_NHC,W19_WOT}. In \cite{PSW24}, the concept of nonlocal $H$-convergence has been used to provide a continuous dependence result for evolutionary equations in the sense of Picard, see, e.g., \cite{Seifert2022}. In \cite[Section 7]{PSW24}, applications have been found in mathematical biology concerning a model related to cell migration and to the PDE system of piezo-elasticity.

For proving the above mentioned compactness statement, we will have the occasion to use a particular version of the div-curl lemma. 
This div-curl lemma used here is based on the abstract findings in \cite{W17_DCL}. These functional analytic variants of the div-curl lemma can also be found in \cite{Pauly2017} with plenty of concrete examples; see also  \cite{Schweizer2018}. We emphasise that in \cite{Pauly2017} one can also find examples with variable coefficients. The idea of the abstract results can be dated back to the notion of compensated compactness (see e.g.~\cite{Tartar2009} and the references therein). 

Both for the Helmholtz decomposition as well as for the solution theory for the electrostatic equations considered here, we refer (for the special case of variable, self-adjoint and coercive coefficients) to the so-called `FA-toolbox' gathered by Pauly and presented for instance in \cite{Pauly2019}; see also \cite{Picard1990} for a perspective to a nonlinear setting. We exemplarily refer to some Banach space versions of the Helmholtz decomposition in \cite{Giga2021TheHD,Maekawa2014,Simader2014}.

The basic compact embedding result for Maxwell's equations is the Picard--Weber--Weck selection theorem, see \cite{Picard1984} and Theorem \ref{thm:PWW} below. The core observation is that compact embedding results of the type
\begin{equation}\label{eq:comp}  \dom(\dive)\cap \dom(\curl^*)\hookrightarrow \hookrightarrow L_2(\Omega)^3
\end{equation}
can be shown \emph{without} the validity of Gaffney's inequality thus allowing to prove this compactness result for more general $\Omega$ such as domains with boundaries that are Lipschitz manifolds (i.e., \textbf{weak Lipschitz domains}). We refer to \cite{Bauer2016} for generalisations to mixed boundary conditions. The proof in \cite{Picard1984} is based on Lipschitz transformations reducing the compactness statement to whether or not the embedding 
\[
  \dom(\dive\varepsilon)\cap \dom(\curl^*) \hookrightarrow L_2(\Omega)^3
\]is compact for self-adjoint, coercive (matrix-valued) multiplication operators $\varepsilon$ and $\Omega$ being the unit cube. This in turn can be shown using weighted scalar products in Hilbert spaces and the classical compact embedding result (that makes use of Gaffney's inequality). We particularly refer to the derivations in \cite[Lemma 5.1 and Section 6.1]{Pauly2021} for these techniques. In particular, one obtains 
\[
  \dom(\dive\varepsilon)\cap \dom(\curl^*) \hookrightarrow\hookrightarrow L_2(\Omega)^3
\]for self-adjoint, coercive $\varepsilon$ and bounded weak Lipschitz domains $\Omega$. It appears that already for constant sequences $(\varepsilon_n)_n=(\varepsilon)_n$ for some nonselfadjoint, coercive $\varepsilon$, the above Theorem \ref{thm:comprough} is new.

Next, we sketch the contents of the sections of the present article. In Section \ref{sec:abstractdiv}, we recall the core finding of \cite{TW14_FE} which forms the foundation of the rationale to follow. The Helmholtz decomposition in the most general form of this article is provided in Section \ref{sec:HD}. This decomposition is used to obtain a solution theory for electrostatics along similar lines as in \cite[Section 4]{Picard1982}, which is presented in Section \ref{sec:estatic}. The concept of nonlocal $H$-convergence is recalled in Section \ref{sec:nHc}. In this section, we will also argue that finite-dimensional harmonic Dirichlet/Neumann fields will not have an effect on the nonlocal $H$-topology. In Section \ref{sec:infdim}, we provide an example for a domain with infinite-dimensional harmonic Dirichlet fields eventually leading to the possibility of an infinite set of mutually incomparable nonlocal $H$-topologies. The nonlocal homogenisation result for electrostatic problems is provided in Section \ref{sec:homel}. Section \ref{sec:energy} contains the convergence of the corresponding stored energy. This result is applied in Section \ref{sec:compact} to derive a proof of Theorem \ref{thm:comprough}, where we also provide an elementary example of nonlocal coefficients. In Section \ref{sec:locnonlocnontriv}, we will show that the well-known notion of local $H$-convergence is equivalent to nonlocal $H$-convergence on multiplication operators for any underlying $\Omega$ satisfying the \textbf{electric compactness property}\textcolor{black}{\footnote{\textcolor{black}{Here we follow the custom of attaching a descriptive word in front of `compactness property' to clarify the contents. In \cite{Bauer2016}, the term Maxwell compactness property was used to describe compact embedding results of the type \eqref{eq:comp} for various sets of boundary conditions. Since we focus here on the `electric boundary condition' (\cite{LS19}), we coined the term `electric compactness condition'.}}}; that is, \eqref{eq:comp}.

We provide a conclusion of our findings in Section \ref{sec:concl}.

For linear operators the term `invertible' will always be reserved for surjective and injective transformations. Since we only consider closed operators defined on Hilbert spaces here; invertible will thus always mean `continuously invertible'. $\K\in \{\R,\C\}$, scalar products are anti-linear in the first and linear in the second component. Direct sums of vector subspaces are denoted by $\dot+$; if these sums happen to be orthogonal, we also use $\oplus$.

\section{Abstract divergence-form operators}\label{sec:abstractdiv}

The Helmholtz decomposition as much as the solution theory for the electrostatic equations considered below is based on \cite[Theorem 3.1]{TW14_FE}. We shortly set the stage for this detour. For a densely defined closed linear operator $C\colon \dom(C)\subseteq \mathcal{H}_0\to \mathcal{H}_1$ acting from the Hilbert spaces $\mathcal{H}_0$ into $\mathcal{H}_1$, we define $H^1(C)\coloneqq (\dom(C),\langle\cdot,\cdot\rangle_C)$ with $\langle\cdot,\cdot\rangle_C$ being the graph scalar product of $C$. We define $C^\diamond \colon \mathcal{H}_1 \to H^{-1}(C)\coloneqq (H^1(C))^*$ via
\begin{equation}\label{eq:Cdi}
    C^\diamond \phi \colon H^1(C) \to \K, u\mapsto \langle \phi, Cu\rangle_{\mathcal{H}_1}.
\end{equation} It is then not difficult to see that $C^\diamond$ extends $C^*$, where we identify $\mathcal{H}_1$ with its dual via the unitary Riesz mapping, see also \cite[Proposition 9.2.2(b)]{Seifert2022}. 

For a Hilbert space $\mathcal{H}$ and a closed subspace $\mathcal{V}\subseteq \mathcal{H}$, we define
\[
   \iota_{\mathcal{V}} \colon \mathcal{V}\hookrightarrow \mathcal{H}, x\mapsto x
\]the canonical embedding. Note that then $\iota_{\mathcal{V}}^* \colon \mathcal{H} \rightarrow \mathcal{V}$ acts as the orthogonal projection onto $\mathcal{V}$ and $\pi_\mathcal{V}\coloneqq \iota_\mathcal{V}\iota_\mathcal{V}^*$ is the actual orthogonal projection; see \cite[Lemma 11.3.3]{Seifert2022}. 

\textcolor{black}{
We recall the closed range theorem, a standard result in functional analysis. 
\begin{proposition}[{{\cite[Theorem IV.1.2]{Goldberg2006}}}]\label{prop:crt} Let $\cH_0,\cH_1$ be Hilbert spaces, $A\colon \dom(A)\subseteq \cH_0\to \cH_1$ closed and densely defined. Then the following conditions are equivalent:
\begin{enumerate}
\item[(i)] $\ran(A)\subseteq \cH_1$ is closed;
\item[(ii)] $\ran(A^*)\subseteq \cH_0$ is closed.
\end{enumerate}
\end{proposition}}

If $C\colon \dom(C)\subseteq \mathcal{H}_0\to \mathcal{H}_1$ is closed and densely defined with closed range $\mathcal{V}_1\coloneqq \ran(C)\subseteq \mathcal{H}_1$, then, by Proposition \ref{prop:crt}, $\mathcal{V}_0 \coloneqq \ran(C^*)\subseteq \mathcal{H}_0$ is closed as well. In this case, we define the \textbf{reduced operator}, 
\[C_{\rdd}\coloneqq \iota_{\mathcal{V}_1}^*C\iota_{\mathcal{V}_0}.\] 

\begin{theorem}[{{\cite[Theorem 3.1]{TW14_FE}}}]\label{thm:TW-Mana} Let $\mathcal{H}_0$, $\mathcal{H}_1$ be Hilbert spaces, $a\in \cB(\mathcal{H}_1)$, $C\colon \dom(C)\subseteq \mathcal{H}_0\to \mathcal{H}_1$ densely defined, closed with closed range $\mathcal{V}_1\coloneqq \ran(C)\subseteq \mathcal{H}_1$. Then the following conditions are equivalent:
\begin{enumerate}
\item[(i)] for all $f\in H^{-1}(C_{\rdd})$ there exists a unique $u\in H^1(C_{\rdd})$ such that
\[
    \langle a C u, C\phi\rangle = f(\phi)\quad(\phi\in H^1(C_{\rdd}));
\]
\item[(ii)] for all $f\in H^{-1}(C_{\rdd})$ there exists a unique $u\in H^1(C_{\rdd})$ such that
\[
     C_{\rdd}^\diamond \iota_{\mathcal{V}_1}^* a\iota_{\mathcal{V}_1} C_{\rdd} u = f;
\]
\item[(iii)] the operator $\iota_{\mathcal{V}_1}^*a\iota_{\mathcal{V}_1} \in \cB(\mathcal{V}_1)$ is continuously invertible.
\end{enumerate}
In either case, if $f\in H^{-1}(C_{\rdd})$ then \[u=(C_{\rdd})^{-1} ( \iota_{\cV_1}^* a\iota_{\cV_1})^{-1} (C_{\rdd}^\diamond)^{-1} f.\]
 If $f\in \ran(C_{\rdd}^*)$ then
\[u=(C_{\rdd})^{-1} ( \iota_{\cV_1}^* a\iota_{\cV_1})^{-1} (C_{\rdd}^*)^{-1} f\]
and the operator
\[
    D_{\rdd, a} \colon \dom(D_{\rdd,a})\subseteq \cV_0 \to \cV_1, u\mapsto C_{\rdd}^\diamond a C_{\rdd} u
\]with domain 
\[
   \dom(D_{\rdd,a}) = \{ u\in \dom(C_{\rdd}) ; a C_{\rdd} u \in \dom(C_{\rdd}^*)\}
\]
is densely defined, closed and continuously invertible. If $\dom(C)\cap \ker(C)^\bot \hookrightarrow \cH_0$ compactly, then $\dom(D_{\rdd,a})$ has compact resolvent. 
 \end{theorem}

\begin{remark}
We recall that the standard assumption for the coefficients $a\in \cB(\mathcal{H}_1)$ for weak variational problems of the type discussed in (a) is that $\Re a\geq c$ for some $c>0$ (i.e.,  $a$ is \textbf{coercive}) in the sense of positive definiteness. A moment's thought reveals that this condition yields $\Re \iota_{\mathcal{V}_1}^* a\iota_{\mathcal{V}_1}\geq c\id_{\mathcal{V}_1}$. In fact, the same holds for $a$ being a maximal monotone relation with $\dom(a)=\mathcal{H}_1$, see \cite{Robinson1999}. This has been exploited in \cite{TW14_FE} to provide a solution theory for nonlinear divergence form problems. 
\end{remark}

\textcolor{black}{Next, we provide a sufficient condition for the closedness of the range of $C$. The corresponding result is similar to the ones in \cite{DTW24} and \cite{Pauly2019}. For convenience, we present the short argument in the present situation.
\begin{proposition}\label{prop:compclos} Let $\mathcal{H}_0$, $\mathcal{H}_1$ be Hilbert spaces, and let $C\colon \dom(C)\subseteq \mathcal{H}_0\to \mathcal{H}_1$ be densely defined, closed. If $H^1(C)\cap \ker(C)^{\bot_{\cH_0}}$ embeds compactly into $\cH_0$, then $\ran(C)\subseteq \cH_1$ is closed.
\end{proposition}
\begin{proof} First we argue, that for closedness of $\ran(C)\subseteq \cH_1$ it suffices to show that there exists $c\geq 0$, such that for all $\phi\in \dom(C)\cap \ker(C)^\bot$ we have
\begin{equation}\label{eq:poinC}
   \|\phi\|_{\cH_0}\leq c\|C\phi\|_{\cH_1}.
\end{equation}
 Indeed, assume that \eqref{eq:poinC} holds, and consider $(\psi_n)_n$ in $\ran(C)$ converging to some $\psi$ in $\cH_1$. For each $n\in \N$, we find $\tilde{\phi}_n \in \dom(C)$ with $\psi_n =C\tilde{\phi}_n$. Decomposing $\tilde{\phi}_n = \phi_{0,n}+\phi_n$ according to $\cH_0 = \ker(C)\oplus \ker(C)^\bot$, as both $\phi_{0,n}, \tilde{\phi}_n\in \dom(C)$, we get ${\phi}_n\in \dom(C)\cap \ker(C)^\bot$ and $C\tilde{\phi}_n=C{\phi}_n=\psi_n$. By \eqref{eq:poinC}, for all $n,m\in \N$, we get
 \[
    \|\phi_n-\phi_m\|_{\cH_0}\leq c\|C\phi_n-C\phi_m\|_{\cH_1}=c\|\psi_n-\psi_m\|_{\cH_1}\to 0\quad (m,n\to\infty).
 \]
 Hence, $(\phi_n)_n$ is a Cauchy-sequence in $H^1(C)$; denote by $\phi \in H^1(C)$ its limit. Thus, from  $C\phi_n =\psi_n$, by continuity of $C\colon H^1(C)\to \cH_1$, it follows that $C\phi = \lim_{n\to\infty} C\phi_n =\lim_{n\to \infty} \psi_n = \psi$. Hence, $\psi = C\phi\in \ran(C)$ and $\ran(C)\subseteq \cH_1$ is closed. \newline
 In order to see that \eqref{eq:poinC} holds, we assume by contradiction that for all $n\in \N$, we find $\phi_n \in \dom(C)\cap \ker(C)^\bot$ such that
 \[
  \|\phi_n\|_{\cH_0}> n\|C\phi_n\|_{\cH_1}.
 \]
 It follows that $\phi_n\neq 0$ and, by possibly dividing the inequality by $\|\phi_n\|_{\cH_0}$, we may assume without loss of generality that $\|\phi_n\|_{\cH_0}=1$ already. Thus, the assumed estimate yields that $C\phi_n\to 0$ as $n\to\infty$ and, in particular, that $(\phi_n)_n$ is a bounded sequence in $H^1(C)\cap \ker(C)^\bot$. By possibly choosing a subsequence, we may assume without restriction that $(\phi_n)_n$ converges weakly in $H^1(C)\cap \ker(C)^\bot$ to some $\phi \in H^1(C)\cap \ker(C)^\bot$. From $C\phi_n\to 0$ and weak continuity of $C\colon H^1(C)\cap \ker(C)^\bot \to \cH_1$, it follows that $\ker(C)^\bot \ni \phi\in \ker(C)$. This yields $\phi=0$. However, $H^1(C)\cap \ker(C)^\bot$ embeds compactly into $\cH_0$, which implies that $(\phi_n)_n$ is strongly convergent in $\cH_0$ and, thus, $1=\|\phi_n\|\to \|\phi\|$ as $n\to\infty$ producing a contradiction to $\phi  =0$. Hence, \eqref{eq:poinC} holds and $\ran(C)\subseteq \cH_1$ is closed.
\end{proof}
We recall the main examples used for the operator $C$ in the applications to follow.
\begin{example}\label{exa:Cex} Let $\Omega\subseteq \R^3$ open. Then 
\begin{align*}
\grad & \colon H^1(\Omega)\subseteq L_2(\Omega)\to L_2(\Omega)^3, \phi\mapsto \nabla \phi \\
\dive & \colon H(\dive,\Omega)\subseteq L_2(\Omega)^3\to L_2(\Omega), \phi\mapsto \nabla\cdot \phi \\
\curl & \colon H(\curl,\Omega)\subseteq L_2(\Omega)^3\to L_2(\Omega)^3, \phi\mapsto \nabla\times \phi
\end{align*}
where all derivatives are computed in the distributional sense, 
\begin{align*}
  H^1(\Omega) & \coloneqq \{ u\in L_2(\Omega); \nabla u \in L_2(\Omega)^3\}, \\
    H(\dive,\Omega) & \coloneqq \{ \phi\in L_2(\Omega)^3; \nabla\cdot u \in L_2(\Omega)^3\}, \\
    H(\curl,\Omega) & \coloneqq\{ \phi\in L_2(\Omega)^3; \nabla\times u \in L_2(\Omega)^3\}.
\end{align*}
Then $\grad, \dive,$ and $\curl$ are closed and densely defined. The first item being easily shown by the closedness of weak derivatives and the second item is straightforward realising that $C_c^\infty(\Omega)$ is dense in $L_2(\Omega)$. Thus, endowing $H^1(\Omega), H(\dive,\Omega), $ and $H(\curl,\Omega) $ with the respective graphs norms of  $\grad, \dive,$ and $\curl$, we can consider these spaces as Hilbert spaces on their own right. Next, we introduce
\begin{align*}
\gradc & \colon H^1_0(\Omega)\subseteq L_2(\Omega)\to L_2(\Omega)^3, \phi\mapsto \nabla \phi \\
\mathring{\dive} & \colon H(\mathring{\dive},\Omega)\subseteq L_2(\Omega)^3\to L_2(\Omega), \phi\mapsto \nabla\cdot \phi \\
\curlc & \colon H(\curlc,\Omega)\subseteq L_2(\Omega)^3\to L_2(\Omega)^3, \phi\mapsto \nabla\times \phi,
\end{align*}
where 
\[
H^1_0(\Omega)\coloneqq \overline{C_c^\infty(\Omega)}^{H^1(\Omega)}, H(\mathring{\dive},\Omega) =\overline{C_c^\infty(\Omega)^3}^{H(\dive,\Omega)}\text{ and }H(\curlc,\Omega)\coloneqq  \overline{C_c^\infty(\Omega)^3}^{H(\curl,\Omega)}.\]
Again, $\gradc$ and $\curlc$ are densely defined and closed. By the definition of the distributional derivative it follows that (see also \cite[Theorem 6.1.2]{Seifert2022})
\begin{equation}\label{eq:adj}
  \dive^* = - \gradc,\quad -\grad^*=-\mathring{\dive} ,\quad   \curl^*=\curlc.
\end{equation}
If the set $\Omega$ is clear from the context, we will also write $H(\dive)$ instead of $H(\dive,\Omega)$, etc. Using the above notation for $C$, we shall in these cases also employ the notation $H^1(\curl)$ instead of $H(\curl, \Omega)$, etc. 
\end{example}
Next, we consider the relevant applications of Proposition \ref{prop:compclos} in the course of the manuscript. More precisely, we find sufficient criteria for $\Omega\subseteq\R^3$ in order that the operators considered in Example \ref{exa:Cex} have closed range. For this,  referring to \cite[Remark 4.7]{PaulyEdd2021}, one can find conditions for the closedness of $\ran(\curl)$. In any case, in order to apply Proposition \ref{prop:compclos}, the following result is helpful.
\begin{theorem}[Picard--Weber--Weck selection theorem, \cite{Picard1984}]\label{thm:PWW} Let $\Omega\subseteq \R^3$ be a weak Lipschitz domain. Then the embeddings
\[
   \dom(\dive)\cap \dom(\mathring{\curl})\hookrightarrow L_2(\Omega)^3\text{ and }   \dom(\mathring{\dive})\cap \dom({\curl})\hookrightarrow L_2(\Omega)^3
\]are compact.
\end{theorem}
\begin{example}\label{ex:rangeclosed} Recall the operators introduced in Example \ref{exa:Cex}.\newline
(a) If $\Omega\subseteq \R^3$ is contained in a slab, then $\ran(\gradc)\subseteq L_2(\Omega)^3$ is closed. Indeed, by Poincar\'e's inequality, we find $c\geq 0$ such that for all $u\in \dom(\gradc)$, 
\[
     \| u \|_{L_2(\Omega)} \leq c\|\grad u \|_{L_2(\Omega)^3}.
\]Thus, an estimate of the form \eqref{eq:poinC} holds with $C=\gradc$, which implies that $\ran(\gradc)$ is closed. By \eqref{eq:adj} and Proposition \ref{prop:crt}, we also get $\ran(\dive)\subseteq L_2(\Omega)^3$ is closed; see also \cite[Proposition 11.3.1 and Corollary 11.3.2]{Seifert2022}.\newline
(b) If $\Omega\subseteq \R^3$ is a bounded weak Lipschitz domain, that is, $\partial\Omega$ is a Lipschitz manifold, then $\ran(\curl)$ and $\ran(\curlc)$ are closed subspaces of $L_2(\Omega)^3$. Indeed, by Theorem \ref{thm:PWW}, we have
\[
   \dom(\curlc)\cap \dom(\dive)\hookrightarrow L_2(\Omega)^3
\] 
compactly. Then, as $\curlc^*=\curl$ (see \eqref{eq:adj}), $\ker(\curlc)^\bot = \overline{\ran}(\curl)$. By $\overline{\ran}(\curl)\subseteq \ker(\dive)\subseteq \dom(\dive)$ (see, e.g., \cite[Proposition 6.1.4]{Seifert2022}), we obtain 
\[
   \dom(\curlc)\cap \ker(\curlc)^\bot\hookrightarrow L_2(\Omega)^3
\] compactly, and, hence, by Proposition \ref{prop:compclos} that $\ran(\curlc)\subseteq L_2(\Omega)^3$ is closed. Finally, by Proposition \ref{prop:crt}, $\ran(\curl)\subseteq L_2(\Omega)^3$ is closed.\newline
(c) Let $a,b,c,d\in \R$ with $a<b$ and $c<d$. Consider $\Omega= (a,b)\times (c,d)\times \R$. Then, by \cite[Example 10]{ABMW19}, $\ran(\curlc)$, and, hence, also, $\ran(\curl)$ are closed subspaces of $L_2(\Omega)^3$.
 \end{example}
 Finally, we comment on the difference between $\curlrd$ and $\curlc$.
 \begin{remark}\label{rem:rdc} (a) Assume that $\Omega\subseteq \R^3$ is such that $\ran(\curl)\subseteq L_2(\Omega)^3$ is closed. We recall from the beginning (or \cite[Proposition 9.2.2(b)]{Seifert2022}) that $\curlc=\curl^*$ is extended by $\curld$; and, similarly, $\curl_{\textnormal{red}}^*\subseteq \curlrd$. Let now $\phi\in \dom(\curlc)$. Then $\phi = \phi_0+\phi_1$ with $\phi_0 \in \ker(\curlc)$ and $\phi_1\in \cV\coloneqq \ran(\curlc^*)=\ran(\curl)$; which yields $\phi_1\in \dom(\curlc)\cap \ker(\curlc)^\bot$.\newline We claim that $\curlc \phi =\curlc \phi_1 = \curlrd \phi_1 = \curlrd \iota_{\cV}^*\phi$: Indeed, since $\curlrd\phi_1 \in H^{-1}(\curlr)$, we need to compare the results as functionals identifying $\cV$ with its dual via the $L_2(\Omega)^3$ scalar product. For this, let $\psi \in   H^1(\curlr)=\dom(\curl)\cap \ker(\curl)^\bot$. We compute
 \begin{multline*}
 (\curlc \phi)(\psi) = \langle \curlc \phi,\psi\rangle_{L_2(\Omega)^3} = \langle  \phi,\curl \psi\rangle_{L_2(\Omega)^3} \\ = \langle  \phi_1,\curl \psi\rangle_{L_2(\Omega)^3} = \langle  \phi_1,\curlr \psi\rangle_{\cV} = (\curlrd \phi_1)(\psi),
 \end{multline*}
  which shows the claim. \newline
    (b) By a slight abuse of notation (deleting $ \iota_{\cV}^*$), we shall write
    \[
         \curlrd \phi \text{ instead of } \curlrd\iota_{\cV}^* \phi.
    \] Thus, part (a) justifies writing
 \[
    \curlc \phi = \curlrd \phi \in H^{-1}(\curlr)
 \]
 also for $\phi \in L_2(\Omega)^3$, where we understand $ \curlc \phi$ as  $ \curlrd \phi$ also if $\phi \in L_2(\Omega)^3\setminus \dom(\curlc)$.
 \end{remark} 
}

\section{The Helmholtz decomposition}\label{sec:HD}

It is the aim of this section to revisit the Helmholtz decomposition for coefficients satisfying the well-posedness criterion in Theorem \ref{thm:TW-Mana}. For this, throughout this section, we let $\Omega\subseteq \R^3$ be open. Next, we \textcolor{black}{quickly recall the} vector analytical operators needed throughout this manuscript from Example \ref{exa:Cex}. We use standard notation from the theory of Sobolev spaces and have $\grad\colon H^1(\Omega)\subseteq L_2(\Omega)\to L_2(\Omega)^3, \phi\mapsto \nabla \phi$ and $\gradc\coloneqq \grad|_{H_0^1(\Omega)}$. Moreover, $-\dive=\gradc^*$ and $\curl\colon H(\curl)\subseteq L_2(\Omega)^3\to L_2(\Omega)^3$ endowed with their respective maximal domains $H(\dive,\Omega)=H(\dive)$ and $H(\curl,\Omega)=H(\curl)$. Recall that $\curl^*= \curlc$, and that, thus, $\curl^*$ is the closure of $C_c^\infty(\Omega)^3$ with respect to the graph norm of $\curl$. In this way, we implement the so-called electric boundary condition, see \ref[eq.~(1.6)]{LS19}.

As a general assumption in this section, we shall ask $\Omega$ to be \textbf{range closed}; that is, both
\[
    \ran(\mathring{\grad})\subseteq L_2(\Omega)^3\text{ and }\ran(\curl)\subseteq L_2(\Omega)^3 \text{ are closed,}
\] where one finds settings for $\Omega$ being range closed in Example \ref{ex:rangeclosed}.

The classical Helmholtz decomposition (see e.g.~\cite{Picard1990}) asserts that we can decompose any $L_2(\Omega)^3$-vector field $E$ with $L_2(\Omega)^3$-orthogonal decomposition in the form
\[
   E = \gradc u + \curl H + x,
\]
where $u\in H_0^1(\Omega)$, $H\in \dom(\curl)\cap \ker(\curl)^\bot$ and $x\in \mathcal{H}_D(\Omega)$ are uniquely determined. Here, 
\[
  \mathcal{H}_D(\Omega)\coloneqq \ker(\dive)\cap \ker(\mathring{\curl})
\]
denotes the space of \textbf{harmonic Dirichlet fields}. Tailored for static Maxwell problems, the Helmholtz decomposition with variable coefficients looks similar and still asserts a direct though in general not orthogonal decomposition. For this, we call $\varepsilon \in \cB(L_2(\Omega)^3)$ \textbf{admissible}, if the following three conditions hold:\vspace*{0.1cm}

\begin{enumerate}
 \item[\textbf{(a1)}] $\varepsilon$ is invertible;\vspace*{0.1cm}
 \item[\textbf{(a2)}] $\iota^*\varepsilon\iota$ is invertible, where $\iota\colon \ran(\gradc)\hookrightarrow L_2(\Omega)^3$;\vspace*{0.1cm}
  \item[\textbf{(a3)}] $\kappa^*\varepsilon^{-1}\kappa$ is invertible, where $\kappa\colon \ran(\curl)\hookrightarrow L_2(\Omega)^3$.\vspace*{0.1cm}
\end{enumerate}

\begin{remark}\label{rem:applMana} (a) The assumption (a2) on $\varepsilon$ implies that Theorem \ref{thm:TW-Mana} applies to $C=\gradc$ (then $C_{\rdd}= \iota^*\gradc$) in order that for all $f\in H^{-1}(\iota^*\gradc)=H^{-1}(\Omega)$ the variational problem
\[
   \langle\varepsilon \gradc u,\gradc u\rangle = f(\phi)\quad(\phi\in H_0^1(\Omega)=\dom(\gradc))
\]admits a unique solution $u\in H_0^1(\Omega)$. In the following we shall refer to this variational problem also by employing the usual distributional divergence $\gradcd$ which extends the above defined $\dive=\gradc^*$, reusing the notion $\dive$ and just writing
\[
   -\dive \varepsilon \gradc u = f.
\]
(b) Quite similarly, the assumption (a1) in conjunction with (a3) enable the usage of Theorem \ref{thm:TW-Mana} for $C=\curl$. Recall that we assumed $\Omega$ to be range closed, particularly implying that $\ran(C)=\ran(\curl)$ is closed. In this case, we obtain $C_{\rdd}=\curlr$ and, thus, $\dom(\curlr)=H(\curl)\cap \ker(\curl)^\bot$. Consequently, due to Theorem \ref{thm:TW-Mana}, for all $g\in H^{-1}(\curlr)$, there exists a unique $v\in H^1(\curlr)$ such that
\[
   \langle\varepsilon^{-1} \curl v,\curl \psi \rangle = g(\psi)\quad(\psi\in H^1(\curlr)).
\]We emphasise that the boundary condition is (weakly) imposed by this variational formulation. Indeed, if $g\in L_2(\Omega)^3$, we would obtain $\varepsilon^{-1}\curl v \in \dom(\curl^*)=\dom(\curlc)$. Hence, the tangential component of $\varepsilon^{-1}\curl v$ vanishes at the boundary (in a generalised sense). 
Similar to the notation in (a) as a reminder of the boundary condition, we sometimes reuse the notation for the adjoint of $\curl^*=\curlc$ to denote $\curld$ (or even write $\curlc$ instead of $\curlrd$, see Remark \ref{rem:rdc} above.
\textcolor{black}{
(c) Even though the coefficients in \cite{Picard1990} can be nonlinear,  admissible coefficients in the present sense complement the ones treated in \cite{Picard1990} as we do not require any monotonicity.}
\end{remark}
The latter remark together with the following shorthands finally puts us in the position to state and prove the generalised Helmholtz decomposition we are aiming for in this section. 
For admissible $\varepsilon$, we define
\[
   \mathcal{H}_\varepsilon \coloneqq \ker(\dive\varepsilon)\cap \ker(\curlc) \text{ and }
 \mathcal{H}^\varepsilon \coloneqq \ker(\dive)\cap \ker(\curlc\varepsilon^{-1}).
\]
\begin{theorem}[generalised Helmholtz decomposition]\label{thm:genHD} Let $\Omega\subseteq \R^3$ be open and range closed. Let $\varepsilon \in \cB(L_2(\Omega)^3)$ be admissible. Then the following decompositions are direct
\begin{align*}
    L_2(\Omega)^3 &= \ran(\gradc) \dot{+} \varepsilon^{-1}\ran({\curl}) \dot{+} \mathcal{H}_\varepsilon \\
    & = \varepsilon\ran(\gradc) \dot{+} \ran({\curl}) \dot{+} \mathcal{H}^\varepsilon.
\end{align*}
Moreover, 
\[
   \varepsilon \colon \mathcal{H}_\varepsilon \to \mathcal{H}^{\varepsilon}, x\mapsto \varepsilon x
\]
is a bijection. 
\end{theorem}
\begin{proof}
Since the rationale for both decompositions is (almost) the same, we only show the first decomposition. We show that this decomposition is direct first. Let $E\in  \ran(\mathring{\grad}) \cap \varepsilon^{-1}\ran({\curl})$. Then we find $u\in H_0^1(\Omega)$ and $H\in \dom(\curl)\cap \ker(\curl)^\bot$ such that
\[
    E = \gradc u = \varepsilon^{-1}\curl H.
\]Applying $\dive \varepsilon$, we infer
\[
   \dive \varepsilon E = \dive \varepsilon \mathring{\grad} u = \dive \curl H =0.
\]Thus, $u=0$, by Theorem \ref{thm:TW-Mana} and (a2) and, hence, $E=0$.

If $E=\ran(\mathring{\grad})\cap \mathcal{H}_\varepsilon$, then we find $u\in H_0^1(\Omega)$ and $x\in \ker(\dive\varepsilon)\cap \ker(\mathring{\curl})$ such that
\[
   E = {\gradc} u = x.
\]Again, applying $\dive\varepsilon$, we obtain
\[
    \dive\varepsilon E = \dive \varepsilon {\gradc} u =\dive\varepsilon x = 0.
\]
Hence, by Theorem \ref{thm:TW-Mana} and (a2), $u=0$ and, thus, $E=0$.

If $E\in \ran(\curl)\cap \mathcal{H}_\varepsilon$, then we find $H\in \dom(\curl)\cap \ker(\curl)^\bot$ and $x\in \ker(\dive\varepsilon)\cap \ker(\mathring{\curl})$ such that
\[
   E = \varepsilon^{-1}\curl H = x. 
\]Applying $\curlc$, we infer
\[
   \curlc E = \curlc \varepsilon^{-1}\curl H = \curl x =0.
\]
Hence, by Theorem \ref{thm:TW-Mana} and (a2), $H=0$, and, therefore, $E=0$.

Finally, let $E\in L_2(\Omega)^3$ we define $u\in H_0^1(\Omega)$ and $H\in \ran(\curl)\cap \ker(\curl)^\bot$ satisfying
\[
   \dive\varepsilon \gradc  u = \dive\varepsilon E\text{ and } \curlc \varepsilon^{-1} \curl H = \curlc\varepsilon^{-1} E.
\]
Then, it follows that
\[
  x\coloneqq E-\mathring{\grad} u -\varepsilon^{-1}\curl H \in \ker(\dive\varepsilon)\cap \ker(\mathring{\curl})=\mathcal{H}_\varepsilon.
\]The last statement in the theorem is an easy consequence of the definition of $\mathcal{H}_\varepsilon$ and $\mathcal{H}^\varepsilon$.
\end{proof}
Theorem \ref{thm:genHD} asserts that the spaces $\ran(\mathring{\grad})$, $\varepsilon^{-1}\ran(\curl)$ and $\mathcal{H}_\varepsilon$ are complemented in $L_2(\Omega)^3$ and as these spaces are closed the projection onto these spaces is continuous. For $\varepsilon\ran(\mathring{\grad})$, $\ran(\curl)$ and $\mathcal{H}^\varepsilon$ a similar observation holds. 

Next, it is the aim to provide some connections of the continuous projections
\begin{equation}\label{eq:peps}
   \pi_\varepsilon \colon \ran(\mathring{\grad}) \dot{+} \varepsilon^{-1}\ran({\curl}) \dot{+} \mathcal{H}_\varepsilon \to \mathcal{H}_\varepsilon
   \end{equation}
   and
   \begin{equation}\label{eq:epsp}
   \pi^\varepsilon \colon \varepsilon\ran(\mathring{\grad}) \dot{+} \ran({\curl}) \dot{+} \mathcal{H}^\varepsilon \to \mathcal{H}^\varepsilon\end{equation}
   to the orthogonal projection
   \[
      \pi_D \colon L_2(\Omega)^3 \to \mathcal{H}_D(\Omega)\coloneqq \ker(\dive)\cap \ker(\mathring{\curl}).
   \]
   \begin{corollary}\label{cor:reformpieps} Let $\varepsilon\in \cB(L_2(\Omega)^3)$ be admissible. Then the following statements hold:
   
   (a) $\pi_D|_{\mathcal{H}_\varepsilon}$ is bijective with inverse given by $\pi_\varepsilon|_{\mathcal{H}_D(\Omega)}$.
   
   (b) Let $E\in L_2(\Omega)^3$. Then
   \begin{align*}
       \pi_\varepsilon\pi_D E &=  (1-\mathring{\grad}(\dive\varepsilon\mathring{\grad})^{-1}\dive\varepsilon)\pi_D E\text{ and } \\
             \pi_D \pi_\varepsilon E &=  (1-\mathring{\grad}(\dive\mathring{\grad})^{-1}\dive)\pi_\varepsilon E.
   \end{align*}    
   \end{corollary}
\begin{proof}
  (a) Let $x_\varepsilon \in \mathcal{H}_\varepsilon$. If $\pi_D x_\varepsilon =0$, then by Theorem \ref{thm:genHD} (for $\varepsilon=\id$) we find $u\in H_0^1(\Omega)$, $v\in H^1(\curlr)$ with
  \[
      x_\varepsilon = \gradc u + \curlr v.
  \]
  Applying $\curlc$, we infer
  \[
     0= \curlc x_\varepsilon = \curlc\curlr v.
  \]Hence, $v=0$. The application of $\dive\varepsilon$ yields
  \[
     0 = \dive\varepsilon x_\varepsilon =\dive\varepsilon \gradc u.
  \] Hence, $u=0$. Thus, $\pi_D|_{\mathcal{H}_\varepsilon}$ is injective. Now, let $x \in \mathcal{H}_D(\Omega)$. Again, by Theorem \ref{thm:genHD} applied in its full generality, we obtain uniquely determined $u\in H_0^1(\Omega)$, $v\in H^1(\curlr)$ and $x_\varepsilon\in \mathcal{H}_\varepsilon$ such that
  \[
     x = \gradc u + \varepsilon^{-1}\curl v +x_\varepsilon.
  \]Then, applying $\curlc$ yields
  \[
     0 = \curlc \varepsilon^{-1}\curl v =\curlrd\kappa^* \varepsilon^{-1}\kappa \curlr v;
  \]thus $v=0$. Hence,
  \[\pi_D (x_\varepsilon) = \pi_D (x-\gradc u) =\pi_D x =x.
\]
Computing the inverse works similarly.

(b) We only show the first equality, the second is proven similarly. We let $\pi_DE=x$. Using the Helmholtz decomposition again, we obtain
\[
   x =\pi_DE = \gradc u_\varepsilon + \varepsilon^{-1}\curl v_\varepsilon +x_\varepsilon 
\]
Applying $\curlc$, we infer
\[
  0=\curlc x = \curlc \varepsilon^{-1}\curl v_\varepsilon
\]
and, hence, $v_\varepsilon=0$ by (a3) and Theorem \ref{thm:TW-Mana}. Next, 
\[
  \dive\varepsilon x = \dive\varepsilon \gradc u_\varepsilon. 
\]
Thus, by (a2) and Theorem \ref{thm:TW-Mana}, $
u_\varepsilon =   (\dive\varepsilon \gradc)^{-1}\dive\varepsilon x$ and so
\begin{align*}
   x &  = \gradc u_\varepsilon + \varepsilon^{-1}\curl v_\varepsilon +x_\varepsilon \\
   & =\gradc (\dive\varepsilon \gradc)^{-1}\dive\varepsilon x  +x_\varepsilon,
\end{align*}which implies the desired equality.
\end{proof}

\begin{remark}\label{rem:posdefadm}If $\varepsilon\in \cB(L_2(\Omega)^3)$ satisfies $\Re \varepsilon\geq c>0$, then $\varepsilon$ satisfies (a1), (a2), and (a3), see \cite[Proposition 6.2.3(b)]{Seifert2022}.
\end{remark}

\begin{proposition} Let $\varepsilon\in \cB(L_2(\Omega)^3)$ satisfy (a1). Assume that $\mathcal{H}_{\textnormal{D}}(\Omega)=\{0\}$. Then (a2) $\iff$ (a3).
\end{proposition}
\begin{proof}
  The assertion follows from the formula stated in Lemma \ref{lem:Schur} below, which in turn follows from a straightforward computation.\end{proof}
  
\begin{lemma}\label{lem:Schur} Let $\cL_0,\cL_1$ be Hilbert spaces and $a=(a_{jk})_{j,k\in\{0,1\}} \in \cB(\cL_0\oplus\cL_1)$ and assume $a_{00}$ is invertible. Then
\[
    \begin{pmatrix} 1 &0 \\  - a_{10}a_{00}^{-1} & 1 \end{pmatrix}  \begin{pmatrix} a_{00} & a_{01} \\ a_{10} & a_{11} \end{pmatrix}
    \begin{pmatrix} 1 & -a_{00}^{-1}a_{01} \\0 & 1 \end{pmatrix}
     =
\begin{pmatrix} a_{00} & 0 \\ 0 & a_{11}-a_{10}a_{00}^{-1}a_{01} \end{pmatrix}.      
\]
   Moreover, the following statements are equivalent:
\begin{enumerate}
   \item[(i)] $a$ invertible;
   \item[(ii)] $a_S \coloneqq a_{11}-a_{10}a_{00}^{-1}a_{01}$ invertible.
   \end{enumerate}
   In either case, we have
   \[
   \begin{pmatrix} a_{00} & a_{01} \\ a_{10} & a_{11} \end{pmatrix}^{-1} =
    \begin{pmatrix} a_{00}^{-1} + a_{00}^{-1}a_{01} a_S^{-1} a_{10} a_{00}^{-1} & - a_{00}^{-1}a_{01}a_S^{-1} \\ -a_{S}^{-1}a_{10}a_{00}^{-1} & a_S^{-1}\end{pmatrix}.
   \]
\end{lemma}

\section{Two formulations for electrostatics}\label{sec:estatic}

    In this section, we present two entirely equivalent formulations of the problem for electrostatics. \textcolor{black}{For these two formulations note that the differential equations themselves correspond to the classical equations for electrostatics with electric boundary conditions, see \cite{FL11,LS19}. Similar to \cite{Picard1982} in order to warrant well-posedness, one needs to require additional information about the harmonic Dirichlet fields. Essentially in the two problem formulations to follow, the equations can be written in either form by suitable substitutions, which depends on the admissible coefficient $\varepsilon$. The details are provided as follows. If there is no risk of confusion, we will also use Remark \ref{rem:rdc}(b) and write $\curlc$ instead of $\curlrd$.}

Similar to the previous section, we let $\Omega\subseteq \R^3$ be open and range closed; that is, we assume that both
\[
   \ran(\gradc)\subseteq L_2(\Omega)^3\text{ and } \ran(\curl)\subseteq L_2(\Omega)^3 \text{ are closed.}
\]
\textcolor{black}{Recall that Example \ref{exa:Cex} provides the definitions for the vector-analytical operators and Example \ref{ex:rangeclosed} contains some examples for $\Omega$ being range closed.} 

Throughout this section, we assume that $\varepsilon\in \cB(L_2(\Omega)^3)$ is admissible; that is, $\varepsilon$ satisfies the conditions (a1), (a2), and (a3) from the previous section. The solution theory for the two problems to be defined next is based on the Helmholtz decomposition presented in Theorem \ref{thm:genHD}. Since we frequently employ the generalised Helmholtz decomposition in the following, we will not mark every instance of application of this theorem.  We will also use the continuous projections $\pi_\varepsilon$ and $\pi^\varepsilon$ introduced in \eqref{eq:peps} and \eqref{eq:epsp}. Finally, we make use of the distribution spaces $H^{-1}(\Omega)$ and $H^{-1}(\curlr)$ mentioned in Remark \ref{rem:applMana}. 
\begin{problem}\label{prb:classical} For given $f\in H^{-1}(\Omega)$, $g\in H^{-1}(\curlr)$, and $x_\varepsilon \in \ker(\dive\varepsilon)\cap \ker(\mathring{\curl})$ find
\[
   E \in L_2(\Omega)^3
\]such that
\begin{equation}\tag{$\mathbb{P}_\varepsilon(f,g,x_\varepsilon)$}\label{prb:class}
   \begin{cases} \dive \varepsilon E = f,&\\
\curlc E=g,&\\
   \pi_{\varepsilon}E=x_\varepsilon.
   \end{cases}
\end{equation}
\end{problem}
\begin{problem}\label{prb:optional} For given $f\in H^{-1}(\Omega)$, $g\in H^{-1}(\curlr)$, and $x^\varepsilon \in \ker(\dive)\cap \ker(\curlc\varepsilon^{-1})$ find
\[
   H \in L_2(\Omega)^3
\]such that
\begin{equation}\tag{$\mathbb{P}^\varepsilon(f,g,x^\varepsilon)$}\label{prb:opt}
   \begin{cases} \dive H= f,&\\
\curlc \varepsilon^{-1}H=g,&\\
   \pi^\varepsilon H =x^\varepsilon.
   \end{cases}
\end{equation}
\end{problem}

We establish well-posedness of either of the above problems:

\begin{theorem}\label{thm:solform} Let $\Omega\subseteq \R^3$ be open and range closed and $\varepsilon\in \cB(L_2(\Omega)^3)$ admissible. Let $f\in H^{-1}(\Omega)$, $g\in H^{-1}(\curlr)$, and $x_\varepsilon\in \mathcal{H}_\varepsilon$ and $x^\varepsilon\in \mathcal{H}^\varepsilon$.

 (a) Problem \ref{prb:classical} is uniquely solvable in $L_2(\Omega)^3$ with continuous solution operator; the solution is given by
\[
   E = \gradc( \dive\varepsilon\gradc )^{-1}f+\varepsilon^{-1}\curlr (\curlrd\varepsilon^{-1}\curlr)^{-1} g +x_\varepsilon.
\]If $f\in L_2(\Omega)$ and $g\in \ran(\curlc)$, then $E\in \dom(\dive\varepsilon)\cap \dom(\curlc)$; with corresponding continuous solution operator.

(b) Problem \ref{prb:optional} is uniquely solvable in $L_2(\Omega)^3$ with continuous solution operator; the solution is given by
\[
   H = \varepsilon{\gradc}( \dive\varepsilon{\gradc} )^{-1}f+\curlr ({\curlrd}\varepsilon^{-1}\curlr)^{-1} g +x^\varepsilon.
\]If $f\in L_2(\Omega)$ and $g\in \ran(\curlc)$, then $H\in \dom(\dive)\cap \dom({\curlc}\varepsilon^{-1})$; with corresponding continuous solution operator.

(c) The Problems \ref{prb:classical} and \ref{prb:optional} are equivalent in the sense that for every solution $E$ solving \eqref{prb:class}, $H=\varepsilon E$ is a solution for \eqref{prb:opt} with $x^\varepsilon =\varepsilon x_\varepsilon $; similarly, if $H$ solves \eqref{prb:opt}, then $E=\varepsilon^{-1}H$ solves \eqref{prb:class} with $x_\varepsilon=\varepsilon^{-1}x^\varepsilon$.
\end{theorem}
\begin{proof}
  (a) We decompose 
  \[
      E = \gradc u+\varepsilon^{-1}\curlr F + q,
  \]
  where $F\in \dom(\curlr)$, $u\in H_0^1(\Omega)$ and $q\in \mathcal{H}_\varepsilon$ are uniquely determined by $E$.
  Next, applying $\dive\varepsilon$ and $\curlc$, we infer
  \[
  f = \dive\varepsilon E = \dive\varepsilon\gradc u\text{ and } g= \curlc E =\curlc\varepsilon^{-1}\curl F
  \]Since, clearly, $q=x_\varepsilon$, we deduce
  \[
     E = \gradc( \dive\varepsilon \gradc )^{-1}f+\varepsilon^{-1}\curlr (\curlrd\varepsilon^{-1}\curlr)^{-1} F +x_\varepsilon,
  \]which is easily seen to be continuous; see in particular Theorem \ref{thm:TW-Mana}. Uniqueness follows from the uniqueness in the Helmholtz decomposition.
  
  (b) For Problem \ref{prb:optional}, we use the other Helmholtz type decomposition:
  \[
  H=  \varepsilon\gradc u+\curl F + q,
  \]for $u\in H_0^1(\Omega)$, $F\in \dom(\curlr)$ and $q\in \mathcal{H}^\varepsilon$ being uniquely determined by $H$.
  
  Next, we apply $\dive$ and $\curlc\varepsilon^{-1}$ to obtain
  \[
     f=\dive\varepsilon \gradc u\text{ and } g = \curlc\varepsilon^{-1}\curl F.
  \]Thus,
  \[
      H= \varepsilon\gradc(\dive\varepsilon \gradc)^{-1}f+\curlr (\curlrd\varepsilon^{-1}\curlr)^{-1} g + x^\varepsilon.
  \]  
  (c) Using the solution formula for $H$, we obtain
  \begin{align*}
    H &=  \varepsilon{\gradc}(\dive\varepsilon {\gradc})^{-1}f+\curlr ({\curlrd}\varepsilon^{-1}\curlr)^{-1} F + x^\varepsilon \\
      & = \varepsilon({\gradc}(\dive\varepsilon {\gradc})^{-1}f+\varepsilon^{-1}\curlr ({\curlrd}\varepsilon^{-1}\curlr)^{-1} F + \varepsilon^{-1}x^\varepsilon)\\
      &  = \varepsilon E.
  \end{align*} Finally, note that $\varepsilon$ bijects $\mathcal{H}_\varepsilon$ onto $\mathcal{H}^\varepsilon$.
  \end{proof}
  
  \begin{corollary}\label{cor:Dfields} Let $E\in L_2(\Omega)^3$, $x_\varepsilon \in \mathcal{H}_{\varepsilon}$.
  
    Then
  \[
      \pi_{\varepsilon} E = x_\varepsilon \iff \pi^{\varepsilon} \varepsilon E = \varepsilon x_\varepsilon.
  \]
  \end{corollary}
\begin{proof}
It follows from Theorem \ref{thm:solform} with $f=\dive\varepsilon E\in H^{-1}(\Omega)$ and $g=\curlc E\in H^{-1}(\curlr)$  \begin{align*}
     E & = \gradc( \dive\varepsilon\gradc )^{-1}f+\varepsilon^{-1}\curlr (\curlrd\varepsilon^{-1}\curlr)^{-1} g +x_\varepsilon \\
      & = \varepsilon^{-1}\big(\varepsilon{\gradc}( \dive\varepsilon {\gradc} )^{-1}f+\curlr (\curlrd\varepsilon^{-1}\curlr)^{-1} g +\varepsilon x_\varepsilon \big)\\
      & = \varepsilon^{-1} H.
   \end{align*}This equality together with the uniqueness statement in the Helmholtz decomposition shows the assertion.
\end{proof}  

\begin{remark}\label{rem:wpmod} With the help of the Corollaries \ref{cor:reformpieps} and \ref{cor:Dfields} it is possible to reformulate Problems \ref{prb:classical} and \ref{prb:optional} in such a way that the data on $\cH_\varepsilon$ and $\cH^\varepsilon$ becomes independent of $\varepsilon$. 

(a) Instead of Problem \ref{prb:classical} one looks for $E\in L_2(\Omega)$ satisfying for given $f\in H^{-1}(\Omega)$ and $g\in H^{-1}(\curlr)$ and $x\in \mathcal{H}_D(\Omega)$ such that
\begin{equation}\label{eq:mod}
  \dive\varepsilon E = f,\ \curlc E = g,\ \text{and}\ \pi_D\pi_\varepsilon E = x.
\end{equation} By Corollary \ref{cor:reformpieps}, the last equation is equivalent to 
\[
   \pi_\varepsilon E = \pi_\varepsilon \pi_D\pi_\varepsilon E = (1-\gradc(\dive \varepsilon \gradc)^{-1}\dive \varepsilon) x\eqqcolon x_\varepsilon.
\]Hence, \eqref{eq:mod} is equivalent to Problem \eqref{prb:class}.

(b) Instead of Problem \ref{prb:optional} one looks for $H\in L_2(\Omega)$ satisfying for given $f\in H^{-1}(\Omega)$ and $g\in H^{-1}(\curlr)$ and $x\in \mathcal{H}_D(\Omega)$ such that
\begin{equation}\label{eq:mod2}
  \dive H = f,\ \curlc\varepsilon^{-1} H = g,\ \text{and}\ \pi_D\pi_\varepsilon \varepsilon^{-1} H =  x.
\end{equation} 
We again reformulate $\pi_D\pi_\varepsilon \varepsilon^{-1} H =  x$. We obtain, by Corollary \ref{cor:reformpieps}, that
\[
 \pi_\varepsilon \varepsilon^{-1} H =   \pi_\varepsilon \pi_D\pi_\varepsilon \varepsilon^{-1} H = (1-\gradc(\dive \varepsilon\gradc)^{-1}\dive\varepsilon)x.
\] Hence, by Corollary \ref{cor:Dfields}, 
\[
 \pi^\varepsilon  H = \varepsilon\pi_\varepsilon \varepsilon^{-1} H=\varepsilon (1-\gradc(\dive \varepsilon\gradc)^{-1}\dive\varepsilon)x\eqqcolon x^\varepsilon.
\]Thus, \eqref{eq:mod2} is equivalent to Problem \eqref{prb:opt}.
\end{remark}
In the next section we recall the core properties of nonlocal $H$-convergence viewed at from an operator-theoretic perspective.

\section{Nonlocal $H$-convergence}\label{sec:nHc}

For nonlocal $H$-convergence -- even though coming from sequences/complexes of closed operators in Hilbert spaces -- the core concept really can be put into general perspective of operators on general Hilbert spaces. One upshot of \cite{W18_NHC} is that the process of homogenisation of linear (operator) coefficients can be framed in terms of an operator topology, which is incomparable to the weak operator topology and strictly weaker than the strong operator and the norm topology. At the end of this section, we provide a proof for Theorem \ref{thm:exunitop}.

Let $\mathcal{H}$ be a Hilbert space, $\mathcal{H}_0 \subseteq \mathcal{H}$ a closed subspace and $\mathcal{H}_1\coloneqq \mathcal{H}_0^\bot$. Using the canonical embeddings $\iota_0 \colon \mathcal{H}_0 \hookrightarrow \mathcal{H}$ and $\iota_1 \colon \mathcal{H}_1 \hookrightarrow \mathcal{H}$, we define the set
\[
   \mathcal{M}(\mathcal{H}_0,\mathcal{H}_1) \coloneqq \{ a\in \cB(\cH ); a_{00}^{-1} \in \cB(\cH_0), a^{-1} \in \cB(\cH) \},
\]
where we abbreviated $a_{jk} \coloneqq \iota_j^* a \iota_k$ for all $j,k\in \{0,1\}$.
The topology for \textbf{nonlocal $H$-convergence (with respect to $(\mathcal{H}_0,\mathcal{H}_1)$)} (or \textbf{Schur topology} or \textbf{nonlocal $H$-topology}), $\tau(\mathcal{H}_0,\mathcal{H}_1)$, is the initial topology induced by the mappings
\begin{align*}
\mathcal{M}(\mathcal{H}_0,\mathcal{H}_1) \ni a & \mapsto a_{00}^{-1}\in \mathcal{B}_{\textnormal{w}}(\mathcal{H}_0) \\
\mathcal{M}(\mathcal{H}_0,\mathcal{H}_1) \ni a & \mapsto a_{00}^{-1}a_{01} \in \mathcal{B}_{\textnormal{w}}(\mathcal{H}_1,\mathcal{H}_0) \\
\mathcal{M}(\mathcal{H}_0,\mathcal{H}_1) \ni a & \mapsto a_{10}a_{00}^{-1} \in \mathcal{B}_{\textnormal{w}}(\mathcal{H}_0,\mathcal{H}_1) \\
\mathcal{M}(\mathcal{H}_0,\mathcal{H}_1) \ni a & \mapsto a_{11}-a_{10}a_{00}^{-1}a_{01} \in \mathcal{B}_{\textnormal{w}}(\mathcal{H}_1).
\end{align*}
 Here, the index $\textnormal{w}$ indicates that the weak operator topology is being used for the operator spaces mentioned.
 
 In order to get a feeling for the topology just introduced, we provide some first examples next. \begin{example}\label{exa:glh}(a) If $\mathcal{H}_0=\{0\}$, then $a_{11}=a$. In this case $\tau(\{0\},\mathcal{H})$ coincides with the weak operator topology on $\mathcal{M}(\{0\},\mathcal{H})=\mathcal{GL}(\mathcal{H})$, the set of invertible operators.

(b) If $\mathcal{H}_0 =\mathcal{H}$, then $\mathcal{H}_1=\{0\}$. In this case $\tau(\mathcal{H},\{0\})$ coincides with the weak operator topology for the inverses on the set $\mathcal{M}(\mathcal{H},\{0\})=\mathcal{GL}(\mathcal{H})$ of invertible linear operators on $\mathcal{H}$.
\end{example}

For topologically trivial $\Omega$, the Schur topology with respect to $(\ran(\gradc),\ran(\curl))$ provides a Hausdorff topology on $\cM(\ran(\gradc),\ran(\curl))$ generalising local $H$-convergence on $L_\infty(\Omega)$-valued matrices; see, in particular, \cite[Theorem 5.11]{W18_NHC}.

\begin{example}[classical nonlocal $H$-convergence]\label{exa:class} Let $\Omega\subseteq \R^3$ be bounded, range closed and assume that $\mathcal{H}_D(\Omega)=\{0\}$; that is, (see \cite{Picard1982,Pauly2021}) $\Omega$ has continuous boundary and connected complement. Then Theorem \ref{thm:genHD} for $\varepsilon=\id$ reads
\[
   L_2(\Omega)^3 = \ran(\gradc)\oplus \ran(\curl).
\]Thus, for $\mathcal{H}=L_2(\Omega)^3$, we can choose 
\[
  \mathcal{H}_0 =\ran(\gradc)\text{ and }\mathcal{H}_1=\ran(\curl).
\]
\end{example}

The aim of this section is to abstractly understand what happens if the assumption $\mathcal{H}_D(\Omega)=\{0\}$ is violated in Example \ref{exa:class}. In other terms, what is the canonical\footnote{`Canonical' in the sense that it is consistent with convergence in the local $H$-sense for multiplication operators.} nonlocal $H$-topology, if
\[
   L_2(\Omega)^3= \ran(\gradc)\oplus \ran(\curl)\oplus \mathcal{H}_D(\Omega)\neq \ran(\gradc)\oplus \ran(\curl) ?
\]
It will turn out that there is not so much ambiguity if $\mathcal{H}_D(\Omega)$ is finite-dimensional. However, if $\mathcal{H}_D(\Omega)$ is infinite-dimensional, many mutually incomparable examples can be constructed. We provide the abstract rationale for this in the next example. The construction of a range closed $\Omega$ with infinite-dimensional space of harmonic Dirichlet fields will be carried out in Section \ref{sec:infdim}, see particularly Example \ref{ex:OmegaRangeClosedID}.

\begin{example}\label{exa:manychoices} (a) Let $\cH_0,\cH_1,\cH_2$ be three Hilbert spaces; $\cH_1$ infinite-dimensional. Consider a sequence of invertible bounded linear operators $(b_n)_n$ on $\cB(\cH_1)$ with $(b_n^{-1})_n$ bounded on $\cB(\cH_1)$ such that
\[
   b_n \to b \in \mathcal{B}_{\textnormal{w}}(\cH_1)
\]
but $(b_n^{-1})_n$ is not convergent in $\mathcal{B}_{\textnormal{w}}(\cH_1)$; such sequences exist as $\cH_1$ infinite-dimensional: Use \cite[Remark 13.2.5 and Proposition 13.2.1 (c)]{Seifert2022} to find two sequences of $(f_n)_n, (g_n)_n$ of strictly positive functions in $L_\infty(\R)$ such that the associated multiplication operators $M_{f_n}$ and $M_{g_n}$ on $L_2(\R)$ satisfy $M_{f_n},M_{g_n} \to \id_{L_2(\R)}$ in the weak operator topology and $M_{f_n}^{-1}=M_{f_n^{-1}} \to \alpha\id_{L_2(\R)}$ as well as $M_{g_n}^{-1}=M_{g_n^{-1}} \to \beta\id_{L_2(\R)}$ in the weak operator topology for some scalars $\alpha\neq \beta$; the desired sequence $b_n$ is then given by $b_n \coloneqq M_{f_{n/2}}$ for $n$ even and $b_n\coloneqq M_{g_{(n-1)/2}}$ for $n$ odd.

In this setting, the operator sequence
\[
(c_n)_n\coloneqq n\mapsto    \begin{pmatrix} \id_{\cH_0} & 0 & 0 \\ 0 & b_n & 0 \\ 0 & 0 & 1_{\cH_2} \end{pmatrix} \in \mathcal{M}
\coloneqq\mathcal{M}(\cH_0,\cH_1\oplus\cH_2)
\cap \mathcal{M}(\cH_0\oplus\cH_1,\cH_2)
\]converges in $\tau_0\coloneqq\tau(\cH_0,\cH_1\oplus\cH_2)$; it does not converge in $\tau_1\coloneqq\tau(\cH_0\oplus \cH_1,\cH_2)$. On the other hand, $(c_n^{-1})_n$ is also a sequence in $\cM$, converging in $\tau_1$ and not converging in $\tau_0$. Thus, $\tau_0$ and $\tau_1$ are incomparable topologies.

(b) Let $\cH$ be a separable, infinite-dimensional Hilbert space; $\cH_0$, $\cH_2$ Hilbert spaces. Let $(\phi_n)_{n\in \N}$ be an orthonormal basis of $\cH$ and let $(N_n)_{n\geq 1}$ be a sequence of pairwise disjoint, infinite subsets of $\N$ such that $\bigcup_n N_n = \N$. Defining $\check{\cH}_n \coloneqq \overline{\lin}\{\phi_k; k\in N_\ell, 1\leq \ell\leq n\}$ and $\hat{\cH}_n \coloneqq {\check{\cH}}_n^\bot\subseteq \cH$, we put
\[
  \tau_{n} \coloneqq \tau(\cH_0\oplus \check{\cH}_n, \hat{\cH}_n \oplus \cH_2).
\]
Then, by (a), for all $n,n'\in \N$ with $n\neq n'$, $\tau_n$ and $\tau_{n'}$ are incomparable.
\end{example}

For finite-dimensional $\mathcal{H}_D(\Omega)$, the, in fact, infinite amount of possible pairwise incomparable topologies does not exist, if we restrict the set of topologies slightly in that we only consider appropriately bounded sequences/nets. This restriction, however, is irrelevant for applications. Before we go into details, we need to mention     another structural property of nonlocal $H$-convergence. For this, note that the Examples \ref{exa:glh} (a) and (b) suggest a symmetry in inversion and simultaneous change of $\cH_0$ and its orthogonal complement. This symmetry holds in the following general sense.

\begin{lemma}\label{lem:inverse} Let $\cH$ be a Hilbert space and $\cH_0\subseteq \cH$ be a closed subspace. Then
\[
    (\cM(\cH_0,\cH_1),\tau(\cH_0,\cH_1) )\ni a \mapsto a^{-1} \in  (\cM(\cH_1,\cH_0),\tau(\cH_1,\cH_0) )
\]
is a homeomorphism. More precisely, we have for all $a\in \cB(\cH)$
\[
   a\in \cM(\cH_0,\cH_1) \iff a^{-1} \in \cM(\cH_1,\cH_0).
\]
In either case, the following equalities hold
\begin{align*}
   [ a^{-1}]_{11}^{-1} & = a_{11}-a_{10}a_{00}^{-1}a_{01} & [a^{-1}]_{11}^{-1}[a^{-1}]_{10} & = -a_{10}a_{00}^{-1} \\
  [a^{-1}]_{01} [a^{-1}]_{11}^{-1} & = -a_{10}a_{00}^{-1} & 
     [ a^{-1}]_{00}-  [a^{-1}]_{01} [a^{-1}]_{11}^{-1} [a^{-1}]_{10} &= a_{00}^{-1}.   
\end{align*}
\end{lemma}
\begin{proof} We show that 
\[
   \cM(\cH_0,\cH_1)  = [\cM(\cH_1,\cH_0)]^{-1}
\]and in passing that
\begin{align*}
   [ a^{-1}]_{11}^{-1} & = a_{11}-a_{10}a_{00}^{-1}a_{01} & [a^{-1}]_{11}^{-1}[a^{-1}]_{10} & = -a_{10}a_{00}^{-1} \\
  [a^{-1}]_{01} [a^{-1}]_{11}^{-1} & = -a_{10}a_{00}^{-1} & 
     [ a^{-1}]_{00}-  [a^{-1}]_{01} [a^{-1}]_{11}^{-1} [a^{-1}]_{10} &= a_{00}^{-1}.   
\end{align*}
These formulas together with the definition of the topology establish all the remaining claims. Now, let $a \in \cB(\cH)$ be continuously invertible. 

Firstly, let us assume that $a \in \cM(\cH_0,\cH_1)$. Then $a_{00}^{-1} \in \cB(\cH_0)$. By Lemma \ref{lem:Schur} it follows that $a_S\coloneqq a_{11}-a_{10}a_{00}^{-1}a_{01}$ is also invertible. The formula for the inverse of $a$ in Lemma \ref{lem:Schur} then shows the desired equalities as well as $a^{-1} \in \cM(\cH_1,\cH_0)$.  On the other hand, if $a^{-1} \in \cM(\cH_1,\cH_0)$, then $ [ a^{-1}]_{11}$ is invertible. As $a^{-1}$ is invertible, by Lemma \ref{lem:Schur} interchanging the roles of $a$ and $a^{-1}$ and $\cH_0$ and $\cH_1$ respectively, we obtain that $[a^{-1}]_S\coloneqq  [ a^{-1}]_{00}-  [a^{-1}]_{01} [a^{-1}]_{11}^{-1} [a^{-1}]_{10}$ is invertible. Moreover, we infer $[a^{-1}]_S^{-1}=a_{00}$ from the concluding formula in Lemma \ref{lem:Schur}. In particular, $a_{00}$ is invertible. Thus, $a\in \cM(\cH_0,\cH_1)$, which concludes the proof.
\end{proof}

Since $\tau(\cH_0, \cH_1)$ is not a locally convex topology, the notion of boundedness is a priori unclear. In the following it will be a desired property for the considered weakly convergent sequences to be bounded. With this in mind, we call a subset $\cA\subseteq \cM(\cH_0,\cH_1)$ \textbf{$\tau(\cH_0,\cH_1)$-bounded}, if all of the following sets 
\begin{align*}
\{a\in \cA; a_{00}^{-1}\} & \subseteq \mathcal{B}(\mathcal{H}_0) \\
\{a\in \cA; a_{00}^{-1}a_{01} \} & \subseteq \mathcal{B}(\mathcal{H}_1,\mathcal{H}_0) \\
\{ a \in \cA;  a_{10}a_{00}^{-1} \} &\subseteq \mathcal{B}(\mathcal{H}_0,\mathcal{H}_1) \\
\{ a \in \cA;  a_{11}-a_{10}a_{00}^{-1}a_{01}\} &\subseteq \mathcal{B}(\mathcal{H}_1)
\end{align*}
 are bounded in the respective norm topologies. The number
 \[
\textnormal{b}(\cA;\tau(\cH_0,\cH_1))\coloneqq \sup_{a\in \cA}     \max\{\|a_{00}^{-1}\|, \|a_{00}^{-1}a_{01}\|, \|a_{10}a_{00}^{-1} \|, \|a_{11}-a_{10}a_{00}^{-1}a_{01}\|\}
 \] 
 is called the \textbf{$\tau(\cH_0,\cH_1)$-bound of $\cA$}. We denote the set of operators in $\cM(\cH_0,\cH_1)$ with a $\tau(\cH_0,\cH_1)$-bound less than some  $\lambda\geq 0$ by $B_{\cM(\cH_0,\cH_1)}(\lambda)$. For the comparison of topologies for different decompositions of $\cH$, we need to restrict our attention to bounded subsets of $\cM(\cH_0,\cH_1)$. 
 
Along the same lines in \cite[Theorem 5.1]{W18_NHC}, one obtains a metrisability result in the separable case; see also \cite[Theorem 5.6]{PSW24}.
 
\begin{theorem}\label{thm:sep} If both $\mathcal{H}_0$ and $\mathcal{H}_1$ are separable, then for all $\lambda\geq 0$, $(B_{\cM(\cH_0,\cH_1)}(\lambda),\tau(\mathcal{H}_0,\mathcal{H}_1))$ is metrisable.
\end{theorem}
 
 To keep the exposition rather lean and to avoid too many technical assumptions in the respective propositions, we introduce $\tau_{\textnormal{b}}(\cH_0,\cH_1)$, the \textbf{ topology of bounded nonlocal $H$-convergence}, which is the inductive topology on $\cM(\cH_0,\cH_1)$ in order that for all $\lambda>0$ the mappings
 \[
     \big( B_{\cM(\cH_0,\cH_1)}(\lambda), \tau(\cH_0,\cH_1)\cap B_{\cM(\cH_0,\cH_1)}(\lambda)\big) \hookrightarrow \cM(\cH_0,\cH_1)
 \]are continuous with the understanding that $\tau(\cH_0,\cH_1)\cap B_{\cM(\cH_0,\cH_1)}(\lambda)$ denotes the relative topology on $B_{\cM(\cH_0,\cH_1)}(\lambda)$ induced by $\tau(\cH_0,\cH_1)$. 
 
 \begin{remark}\label{rem:inverse} A similar statement to the one for the inversion of operators in Lemma \ref{lem:inverse} is also true for $\tau_\textnormal{b}$. In fact we have
  $B_{\cM(\cH_0,\cH_1)}(\lambda)=B_{[\cM(\cH_0,\cH_1)]^{-1}}(\lambda)\coloneqq B_{\cM(\cH_1,\cH_0)}(\lambda)$ for all $\lambda\geq 0$. In consequence, 
\[
   (\cM(\cH_0,\cH_1),\tau_{\textnormal{b}}(\cH_0,\cH_1) )\ni a \mapsto a^{-1} \in  (\cM(\cH_1,\cH_0),\tau_{\textnormal{b}}(\cH_1,\cH_0) )
\] is a homeomorphism.
 Indeed, this statement is a straightforward consequence of the formulas presented in Lemma \ref{lem:inverse}.
 \end{remark}

\begin{proposition}\label{prop:charbd} Let $Y$ be a topological space and $g\colon \cM(\cH_0,\cH_1)\to Y$. Then the following conditions are equivalent:
\begin{enumerate}
\item[(i)] $g$ is continuous;
\item[(ii)] for all $\lambda\geq 0$, $g_\lambda \colon \big( B_{\cM(\cH_0,\cH_1)}(\lambda), \tau(\cH_0,\cH_1)\cap B_{\cM(\cH_0,\cH_1)}(\lambda)\big) \to Y, a\mapsto g(a)$ is continuous;
\item[(iii)] for all $\tau(\cH_0,\cH_1)$-bounded convergent nets $(a^{(\iota)})_\iota$  to some $a\in \cM(\cH_0,\cH_1)$, one has $g(a^{(\iota)})\stackrel{\iota}{\to} g(a)$.
\end{enumerate}
If, in addition, $\cH$ is separable, then the following condition is also equivalent to (i), (ii) and (iii):
\begin{enumerate}
\item[(iv)] for all $\tau(\cH_0,\cH_1)$-convergent sequences $(a^{(n)})_n$  to some $a\in \cM(\cH_0,\cH_1)$, $g(a^{(n)})\to g(a)$ as $n\to\infty$.
\end{enumerate}
\end{proposition}
\begin{proof}
The conditions (i) and (ii) being equivalent follow from the general fact for final topologies, see e.g.~\cite[Theorem 10.1]{Voigt2020} in the context of topological vectors spaces. (ii) and (iii) are a reformulation of one another using the fact that also $a \in B_{\cM(\cH_0,\cH_1)}(\lambda)$ if $a_\iota\in B_{\cM(\cH_0,\cH_1)}(\lambda)$ for all $\iota$ and some $\lambda\geq0$.

If $\cH$ is separable, then $\big( B_{\cM(\cH_0,\cH_1)}(\lambda), \tau(\cH_0,\cH_1)\cap B_{\cM(\cH_0,\cH_1)}(\lambda)\big)$ is metrisable Theorem \ref{thm:sep}. Hence, continuity of $g_\lambda$ is characterised by sequential continuity; thus (iv) implies (ii). If, on the other hand, (ii) holds then for any $\tau(\cH_0,\cH_1)$-convergent sequence $(a^{(n)})_n$ sequence with limit $a$ we find $\lambda\geq 0$ such that $a^{(n)},a\in  B_{\cM(\cH_0,\cH_1)}(\lambda)$ employing the uniform boundedness principle. Thus, (iii) holds due to continuity of $g_\lambda$.
\end{proof}

\begin{theorem}\label{thm:finite} Let $\mathcal{H}$ be a Hilbert space and $\mathcal{H}_0\subseteq \mathcal{K}_0 \subseteq \mathcal{H}$ be closed subspaces. Assume that $\mathcal{K}_0 \ominus \mathcal{H}_0$ is finite-dimensional. Then $\tau_{\textnormal{b}}(\mathcal{H}_0,\cH_0^\bot)=\tau_{\textnormal{b}}(\mathcal{K}_0,\cK_0^\bot)$ on $\cM(\cH_0,\cH_0^\bot)\cap\cM(\cK_0,\cK_0^\bot)$.
\end{theorem}
\begin{proof} We put $\cH_1\coloneqq \cH_0^\bot$ and $\cK_1\coloneqq \cK_0^\bot$.
It suffices to show that with $\cM\coloneqq \cM(\cH_0,\cH_1)\cap\cM(\cK_0,\cK_1)$ the inclusion
\begin{equation}\label{eq:ast}
(\cM,\tau_{\textnormal{b}}(\mathcal{H}_0,\cH_1)) \hookrightarrow (\cM, \tau_{\textnormal{b}}(\mathcal{K}_0,\cK_1))
\end{equation}
is continuous. Indeed, the other inclusion is then a consequence of this statement in conjunction with Remark \ref{rem:inverse} and Lemma \ref{lem:inverse}.

For the continuity of the above inclusion, we denote $\cL_0\coloneqq \cH_0, \mathcal{L}_1\coloneqq \mathcal{K}_0\ominus\mathcal{H}_0$ and $\cL_2 \coloneqq \cK_0^{\bot}$. Then $\cH=\cL_0\oplus\cL_1\oplus \cL_2$. For every $a\in \cB(\cH)$ we use block operator notation and suggestive indices to write $a$ as an operator in $\cB(\cL_0\oplus\cL_1\oplus \cL_2)$:
\[
  \cB(\cH)\ni  a\sim \begin{pmatrix} a_{00}& a_{01}& a_{02} \\ a_{10}& a_{11}& a_{12} \\ a_{20}& a_{21}& a_{22} \end{pmatrix} \in \cB(\cL_0\oplus\cL_1\oplus \cL_2).
\]With this identification, $\tau(\mathcal{H}_0,\mathcal{H}_1)$ is induced by the mappings
\begin{align*}
 a &\mapsto a_{00}^{-1} \\
 a & \mapsto \begin{pmatrix} a_{10} \\ a_{20} \end{pmatrix}a_{00}^{-1} = \begin{pmatrix} a_{10}a_{00}^{-1} \\ a_{20}a_{00}^{-1} \end{pmatrix} \\
  a & \mapsto a_{00}^{-1}\begin{pmatrix} a_{01} & a_{02} \end{pmatrix} =
\begin{pmatrix}   a_{00}^{-1}a_{01} &   a_{00}^{-1}a_{02} \end{pmatrix}
  \\
    a & \mapsto \begin{pmatrix} a_{11} & a_{12} \\ a_{21} & a_{22} \end{pmatrix} - \begin{pmatrix} a_{10} \\ a_{20} \end{pmatrix} a_{00}^{-1}\begin{pmatrix} a_{01} & a_{02} \end{pmatrix} = \begin{pmatrix}a_{11} - a_{10}a_{00}^{-1}a_{01} & a_{12}-a_{10}a_{00}^{-1}a_{02}\\ a_{21} - a_{20}a_{00}^{-1}a_{01} & a_{22} - a_{20}a_{00}^{-1}a_{02} \end{pmatrix}.
\end{align*}
$\tau(\mathcal{K}_0,\mathcal{K}_1)$ is induced by 
\begin{align*}
 a &\mapsto \begin{pmatrix} a_{00} & a_{01} \\ a_{01} & a_{11} \end{pmatrix}^{-1} \\
 a & \mapsto  \begin{pmatrix} a_{20} & a_{21} \end{pmatrix}\begin{pmatrix} a_{00} & a_{01} \\ a_{01} & a_{11} \end{pmatrix}^{-1} \\
  a & \mapsto \begin{pmatrix} a_{00} & a_{01} \\ a_{01} & a_{11} \end{pmatrix}^{-1} \begin{pmatrix} a_{02} \\ a_{12} \end{pmatrix}  \\
    a & \mapsto a_{22}-\begin{pmatrix}a_{20} & a_{21} \end{pmatrix} \begin{pmatrix} a_{00} & a_{01} \\ a_{01} & a_{11} \end{pmatrix}^{-1} \begin{pmatrix} a_{02}\\ a_{12} \end{pmatrix} 
\end{align*} Using the formula for the inverse in Lemma \ref{lem:Schur} and $a_S \coloneqq a_{11}-a_{10}a_{00}^{-1}a_{01}$ , we obtain
\begin{align*}
  &  \begin{pmatrix} a_{00} & a_{01} \\ a_{10} & a_{11} \end{pmatrix}^{-1}
 = \left(  \begin{pmatrix} 1 &0 \\   a_{10}a_{00}^{-1} & 1 \end{pmatrix} 
 \begin{pmatrix} a_{00} & 0 \\ 0 & a_{S} \end{pmatrix}
  \begin{pmatrix} 1 & a_{00}^{-1}a_{01} \\0 & 1 \end{pmatrix}
 \right)^{-1} \\
 & =    \begin{pmatrix} 1 & -a_{00}^{-1}a_{01} \\0 & 1 \end{pmatrix}
 \begin{pmatrix} a_{00}^{-1} & 0 \\ 0 & a_{S}^{-1}\end{pmatrix}
  \begin{pmatrix} 1 &0 \\   -a_{10}a_{00}^{-1} & 1 \end{pmatrix} 
 \\
& =  \begin{pmatrix} a_{00}^{-1} + a_{00}^{-1}a_{01} a_S^{-1} a_{10} a_{00}^{-1} & - a_{00}^{-1}a_{01}a_S^{-1} \\ -a_{S}^{-1}a_{10}a_{00}^{-1} & a_S^{-1}\end{pmatrix} \end{align*} and, thus, $\tau(\mathcal{K}_0,\mathcal{K}_1)$ is induced by 
\begin{align*}
 a &\mapsto \begin{pmatrix} a_{00}^{-1} + a_{00}^{-1}a_{01} a_S^{-1} a_{10} a_{00}^{-1} & - a_{00}^{-1}a_{01}a_S^{-1} \\ -a_{S}^{-1}a_{10}a_{00}^{-1} & a_S^{-1}\end{pmatrix}  \\
 & = \begin{pmatrix} a_{00}^{-1} + a_{00}^{-1}a_{01} \times \id_{1} \times a_S^{-1} \times \id_{1} \times a_{10} a_{00}^{-1} & - a_{00}^{-1}a_{01} \times \id_{1} \times a_S^{-1} \\ -a_{S}^{-1} \times \id_{1} \times a_{10}a_{00}^{-1} & a_S^{-1}\end{pmatrix}  \\ 
 a & \mapsto \begin{pmatrix} a_{20}a_{00}^{-1} + a_{20}a_{00}^{-1}a_{01} a_S^{-1} a_{10} a_{00}^{-1} + a_{21}( -a_{S}^{-1}a_{10}a_{00}^{-1})& - a_{20}a_{00}^{-1}a_{01}a_S^{-1} +a_{21} a_S^{-1}\end{pmatrix} \\
  & = \begin{pmatrix} a_{20}a_{00}^{-1} + (a_{20}a_{00}^{-1}a_{01}-a_{21}) \times \id_{1} \times a_S^{-1} \times \id_{1} \times a_{10} a_{00}^{-1}& (- a_{20}a_{00}^{-1}a_{01} +a_{21})\times \id_{1} \times a_S^{-1}\end{pmatrix} \\
  a & \mapsto \begin{pmatrix} a_{00}^{-1}a_{02} + a_{00}^{-1}a_{01} a_S^{-1} a_{10} a_{00}^{-1}a_{02}  - a_{00}^{-1}a_{01}a_S^{-1}a_{12} \\ -a_{S}^{-1}a_{10}a_{00}^{-1}a_{02} + a_S^{-1}a_{12}\end{pmatrix}   \\
  & = \begin{pmatrix} a_{00}^{-1}a_{02} + a_{00}^{-1}a_{01} \times \id_{1} \times a_S^{-1} \times \id_{1} \times ( a_{10} a_{00}^{-1}a_{02}  -a_{12}) \\ a_{S}^{-1}
  \times \id_{1} \times (-a_{10}a_{00}^{-1}a_{02} + a_{12})\end{pmatrix}   \\
    a & \mapsto a_{22}-\begin{pmatrix}a_{20} & a_{21} \end{pmatrix} \begin{pmatrix} a_{00}^{-1} + a_{00}^{-1}a_{01} a_S^{-1} a_{10} a_{00}^{-1} & - a_{00}^{-1}a_{01}a_S^{-1} \\ -a_{S}^{-1}a_{10}a_{00}^{-1} & a_S^{-1}\end{pmatrix}\begin{pmatrix} a_{02}\\ a_{12} \end{pmatrix} \\
            & = (a_{22}-a_{20}a_{00}^{-1}a_{02}) - a_{20}a_{00}^{-1}a_{01} \times \id_{1} \times a_S^{-1} \times \id_{1} \times ( a_{10} a_{00}^{-1}a_{02}  -a_{12}) \\ & \quad \quad- a_{21}a_{S}^{-1}
  \times \id_{1} \times (-a_{10}a_{00}^{-1}a_{02} + a_{12}) \\
       & = (a_{22}-a_{20}a_{00}^{-1}a_{02}) + (a_{21}- a_{20}a_{00}^{-1}a_{01}) \times \id_{1} \times a_S^{-1} \times \id_{1} \times ( a_{10} a_{00}^{-1}a_{02}  -a_{12}),
\end{align*}
where $\id_1$ denotes the identity on $\cL_1$. Since $\cL_1$ is finite-dimensional, $\id_1$ is a compact operator. Moreover, in finite dimensions, the weak operator topology coincides with the norm topology. Thus, inversion is a continuous process for invertible operators on $\cL_1$ in the weak operator topology. In order to show continuity of \eqref{eq:ast}, it suffices to use the net characterisation in Proposition \ref{prop:charbd} (iii), in combination with the result on weakly converging products in Lemma \ref{lem:prodweak} below. Note that the above expressions of the mappings inducing the topology $\tau(\cK_0,\cK_1)$ are written as sums and products of mappings inducing the topology on $\tau(\cH_0,\cH_1)$ and the compact operator $\id_1$. This eventually shows the assertion. 
\end{proof}


\begin{lemma}\label{lem:prodweak} Let $n\in\N$, $\mathcal{H}_0, \ldots, \cH_{n+1}$ Hilbert spaces, $(A_\iota^{(\ell)})_{\iota}$, $\ell\in\{0,\ldots,n\}$ bounded nets in $\cB(\cH_{\ell},\cH_{\ell+1})$ respectively converging in the weak operator topology to some $A^{(\ell)}\in \cB(\cH_{\ell},\cH_{\ell+1})$. For $\ell\in \{1,\ldots, n\}$ let $T_\ell \in \cB(\cH_\ell)$ be compact. Then
\[
      A_\iota^{(n)}T_{n} A_\iota^{(n-1)} T_{n-1} \cdots   A_\iota^{(1)} T_1 A_\iota^{(0)} \to 
       A^{(n)}T_{n} A^{(n-1)} T_{n-1} \cdots   A^{(1)} T_1 A^{(0)}
\]in the weak operator topology of $\cB(\cH_{0},\cH_{n+1})$.
\end{lemma}
\begin{proof}
  We prove the claim by induction on $n$. For $n=0$, there is nothing to show. For the inductive step, we recall that compact operators map weakly convergent sequences to norm convergent sequences. Thus, assuming the claim is true for some $n-1$, we let $x \in \cH_0$ and $y\in \cH_{n+1}$. Then, by the induction hypothesis, $z_\iota\coloneqq A_\iota^{(n-1)} T_{n-1} \cdots   A_\iota^{(1)} T_1 A_\iota^{(0)} x$ weakly converges to $z\coloneqq A^{(n-1)} T_{n-1} \cdots   A^{(1)} T_1 A^{(0)} x$ as $\iota\to\infty$. The aim is to show that $A_\iota^{(n)} T_n z_\iota$ weakly converges to $A^{(n)} T_n z$ as $\iota\to\infty$. As $T_n$ is compact and $(z_\iota)_\iota$ is bounded, we obtain $T_n z_\iota \to T_n z$ in norm, as $\iota\to\infty$. Since computing the adjoint is continuous in the weak operator topology, we infer for $\iota$:
  \[
     \langle A_\iota^{(n)} T_n z_\iota,y\rangle_{\cH_{n+1}} =
          \langle  T_n z_\iota,(A_\iota^{(n)})^*y\rangle_{\cH_{n}} \to 
                    \langle  T_n z,(A^{(n)})^*y\rangle_{\cH_{n}} =
                                        \langle  A^{(n)} T_n z,y\rangle_{\cH_{n+1}}
  \]as $\iota \to\infty$. This completes the proof.
\end{proof}

\begin{proof}[Proof of Theorem \ref{thm:exunitop}] At first note that for all $\Omega\subseteq \R^3$ open and range closed and $0<\alpha<\beta$, the set $\cM(\alpha,\beta)$ is a $\tau(\cH,\cH^\bot)$-bounded subset for all closed subspaces $\cH\subseteq L_2(\Omega)^3$. 

Next, we prove statement (a). Let $\Omega\subseteq \R^3$ open and range closed and let $\cH_0\subseteq L_2(\Omega)^3$ be a closed subspace with the properties
\[
   \ran(\gradc)\subseteq \cH_0\text{ and }\ran(\curl)\subseteq \cH_0^\bot.
\]
Then, any $\cH_0$ chosen in such a way defines an element in $\mathfrak{T}$. Note that $\tau_\Omega$ is uniquely determined, if $\cH_D(\Omega)=\{0\}$ since then necessarily $\cH_0=\ran(\gradc)$ and $\tau_\Omega$ is the classical nonlocal $H$-convergence topology on the $\tau_\Omega$-bounded set $\cM(\alpha,\beta)$, see also Example \ref{exa:class}; proving (a).

(b) The uniqueness statement follows from Theorem \ref{thm:finite}. Indeed, let $\cH_0$ be as above. Then $\cH_0\ominus \ran(\gradc)$ is finite-dimensional since $\cH_0^\bot \supseteq \ran(\curl)$ and $\cH_D(\Omega)$ is finite-dimensional. Thus, by Theorem \ref{thm:finite}, we obtain 
\[
   \tau_{\textnormal{b}}(\ran(\gradc),\ran(\gradc)^\bot)=   \tau_{\textnormal{b}}(\cH_0,\cH_0^\bot),
\]which yields the assertion in conjunction with the observation from the beginning of this proof that $\cM(\alpha,\beta)$ is both $\tau(\ran(\gradc),\ran(\gradc)^\bot)$-bounded and $\tau(\cH_0,\cH_0^\bot)$-bounded.
\end{proof}

\section{Infinite-dimensional harmonic Dirichlet fields}\label{sec:infdim}

In this section we aim to construct a range closed, open $\Omega\subseteq \R^3$ such that $\mathcal{H}_D(\Omega)$ is infinite-dimensional. This will then provide a particular example proving that there exists infinitely many pairwise incomparable topologies all providing in some sense an extension of the classical nonlocal $H$-convergence topology:
\begin{example}\label{exa:omeginfi} (a) Let $\Omega\subseteq \R^3$ be open and range closed with $\mathcal{H}_D(\Omega)$ infinite-dimensional (to be constructed below). Consider the topologies $(\tau_{k})_k$ in Example \ref{exa:manychoices} (b) with $\ran(\gradc)=\cH_0$, $\ran(\curl)=\cH_2$ and $\mathcal{H}_D(\Omega)=\cH$. Let $(a_n)_n$ be an admissible sequence in $\cB(L_2(\Omega)^3)$. Assume all $a_n$ have the following common block structure
\[
    \begin{pmatrix}
      * & 0 & * \\ 0 & * & 0 \\ * & 0 & * 
    \end{pmatrix} \in \cB(\ran(\gradc)\oplus \mathcal{H}_D(\Omega) \oplus \ran(\curl)),
\] where $*$ symbolises potential non-zero entries.  Let $f\in H^{-1}(\Omega)$ and $g\in H^{-1}(\curlr)$ and $u_n \in H_0^1(\Omega)$ and $v_n\in H^1(\curlr)$ satisfying
\[
    -\dive a_n \gradc u_n = f,\quad \curlc a_n^{-1} \curl v_n =g.
\]Now, let $a_n\to a$ in $\tau_k$ for some admissible $a$ and $k\geq 1$. Then using Theorem \ref{thm:TW-Mana} it is not difficult to see that $u_n, a_n \gradc u_n , v_n, a_n^{-1} \curlc v_n$ all weakly converge to the respective weak limits $u, a \gradc u , v, a_n^{-1} \curlc v$ with $u,v$ satisfying
\[
  -\dive a \gradc u = f,\quad \curlc a^{-1} \curl v =g.
\]However, by Example \ref{exa:manychoices} (b) there are choices for $(a_n)_n$, where $a$ is different for different $k$. 
%
\end{example}

We shall now provide an abstract detour to construct the $\Omega$ needed in Example \ref{exa:omeginfi}. 
Throughout, let $\cH_0,\cH_1,\cH_2$ be Hilbert spaces. For a densely defined, closed linear operator
\[
   A\colon \dom(A)\subseteq \cH_0\to \cH_1,
\]
we define
\[
    \gamma(A)\coloneqq \inf_{\phi\in \dom(A)\cap\ker(A)^\bot}\frac{\|A\phi\|_{\cH_1}}{\|\phi\|_{\cH_0}},
\]
the \textbf{minimum modulus of $A$}.
The number $\gamma(A)$ is decisive when it comes to decide whether or not the range of $A$ is closed.

\begin{theorem}[{{\cite[Theorem IV.1.6]{Goldberg2006}}}]\label{thm:gara} Let $A\colon \dom(A)\subseteq \cH_0\to \cH_1$ be a closed linear operator. Then the following conditions are equivalent:
\begin{enumerate}
\item[(i)] $\ran(A)\subseteq \cH_1$ closed;
\item[(ii)] $\gamma(A)>0$.
\end{enumerate}
\end{theorem}
The next observation provides an estimate for the minimum modulus if $A$ is given as an orthogonal sum of many operators.
\begin{proposition}\label{prop:osum}
For $n\in \N$, let $A^{(n)} \colon \dom(A^{(n)})\subseteq \cH_0 \to \cH_{1}$ be densely defined, closed and linear. We define
\[
   \mathcal{A}\coloneqq \bigoplus_{n\in \N} A^{(n)} \colon \dom(   \mathcal{A})\subseteq \ell_2(\N;\cH_0)\to \ell_2(\N;\cH_{1}),\quad (\phi_n)_n\mapsto (A^{(n)}\phi_n)_n
\]
where 
\[
   \dom(   \mathcal{A}) = \{ (\phi_n)_n \in \ell_2(\N;\cH_0); \phi_n \in \dom(A^{(n)})\text{ and }(A^{(n)}\phi_n)_n \in \ell_2(\N;\cH_{1})\}.
\]
Then $\mathcal{A}$ is densely defined, closed, linear operator with $\mathcal{A}^*= \bigoplus_{n\in \N} (A^{(n)})^*$ and
\[
   \ker(\mathcal{A})= \bigoplus_{n\in \N} \ker(A^{(n)}).
\]
Moreover, if $\gamma(A^{(n)})>0$, then
\[
    \inf_{n\in \N} \gamma(A^{(n)})\leq \gamma(\mathcal{A}).
\]
\end{proposition}
\begin{proof}
Most of the statements being elementary and following from corresponding statements for (operator-valued, abstract) multiplication operators, we only prove the last inequality. Let $\phi=(\phi_n)_n$ in $\dom(\mathcal{A})\cap \ker(\mathcal{A})^\bot$. By the formula for the kernel of $\mathcal{A}$ we obtain, for all $n\in \N$, $\phi_n\in \dom(A^{(n)})\cap \ker(A^{(n)})^\bot$. Let $\alpha>1$. Then, for $n\in \N$,
\[
    \frac{\gamma(A^{(n)})}{\alpha} \leq \frac{\|A^{(n)}\phi_n\|_{\cH_1}}{\|\phi_n\|_{\cH_0}}; \text{ and, thus, }    \big(\frac{\gamma(A^{(n)})}{\alpha}\big)^2\|\phi_n\|_{\cH_0}^2 \leq \|A^{(n)}\phi_n\|_{\cH_1}^2.
\]Hence, summing over $n\in \N$, we infer
\[
\frac{1}{\alpha^2} \inf_{n\in \N} \gamma(A^{(n)})^2\|\phi\|^2_{\ell_2(\N;\cH_0)}\leq \|\mathcal{A}\phi\|_{\ell_2(\N;\cH_1)}^2.
\]
Therefore, 
\[
   \frac{1}{\alpha} \inf_{n\in \N} \gamma(A^{(n)}) \leq \gamma(\mathcal{A}).
\]Letting $\alpha\to 1$, we deduce the asserted inequality.
\end{proof}
\begin{theorem}\label{thm:osumscomplex} For $n\in \N$, $j\in \{0,1\}$, let $A_j^{(n)} \colon \dom(A_j^{(n)})\subseteq \cH_j \to \cH_{j+1}$ be densely defined, closed linear operators satisfying $\ran(A_0^{(n)})\subseteq \ker(A_1^{(n)})$.
Then for $\mathcal{A}_j \coloneqq \bigoplus_{n\in\N}A_j^{(n)}$, $j\in \{0,1\}$ we have $\ran(\cA_0)\subseteq \ker(\cA_1)$ and 
\[
   \ker(\mathcal{A}_1)\cap \ker(\mathcal{A}_0^*) = \bigoplus_{n\in \N} \ker(A_1^{(n)})\cap \ker ((A_0^{(n)})^*)\subseteq \ell_2(\N;\cH_1).
\]
Moreover, if for $j\in \{0,1\}$
\[
    \inf_{n\in \N} \gamma(A_j^{(n)})>0,
\]then $\ran(\cA_0), \ran(\cA_1)$ are closed.
\end{theorem}
\begin{proof}
The statements are almost direct consequences of Proposition \ref{prop:osum}. The closedness of the ranges follows from Theorem \ref{thm:gara} in conjunction with the inequality asserted in Proposition \ref{prop:osum}.\end{proof}

We are now in the position to provide an example for the desired range closed domain $\Omega$ with $\cH_D(\Omega)$ infinite-dimensional.
\begin{example}\label{ex:OmegaRangeClosedID} Let $\Omega\subseteq \R^3$ be open, bounded and range closed with $\mathcal{H}_D(\Omega)\neq \{0\}$\footnote{As discussed in Figure \ref{fig:hDN}, an annular region is a possible choice.}. Then by Theorem \ref{thm:osumscomplex}, $\N$ identical copies of the operators $\gradc$ and $\curlc$ yield 2 operators $\mathcal{A}_0$, $\mathcal{A}_1$ with closed range and infinite-dimensional $ \ker(\mathcal{A}_1)\cap \ker(\mathcal{A}_0^*)$. We note that this orthogonal sum, by translation invariance of the derivative (and the locality of the operators involved), can be realised as a single set $\hat\Omega\subseteq \R^3$, which is a grid of copies of $\Omega$ that are arranged in such a way they do not intersect; see also Figure \ref{fig:infdeg}.
\begin{figure}
\begin{minipage}{\textwidth}
\centering
 \tdplotsetmaincoords{0}{0} 
\begin{tikzpicture}[tdplot_main_coords, scale=0.5]
  \def\sphereRadius{0.6}
  \def\gridSeparation{2.5}

  \foreach \x in {-6,...,6} {
    \foreach \y in {-2,...,2} {
      \pgfmathsetmacro{\shiftX}{\x*\gridSeparation}
      \pgfmathsetmacro{\shiftY}{\y*\gridSeparation}
      
      \shade[ball color=black!20, opacity=1] (\shiftX, \shiftY, 0) circle (1);
      
      \begin{scope}
        \clip (\shiftX, \shiftY) circle (\sphereRadius);
        \shade[ball color=white] (\shiftX, \shiftY, 0) circle (\sphereRadius);
      \end{scope}
    }
  }
\end{tikzpicture}
\end{minipage}
\caption{Part of a range closed topologically nontrivial domain $\Omega$ with infinite-dimensional space of harmonic Dirichlet fields; each individual annular region is arranged on an infinite discrete grid.}\label{fig:infdeg}
\end{figure}
\end{example}

\section{Homogenisation theorems}\label{sec:homel}

In this section, we combine the results concerning well-posedness of electrostatic equations and on the topology of nonlocal $H$-convergence in the case where it can be guaranteed that $\mathcal{H}_D(\Omega)$ is finite-dimensional, only. Here, we let $\Omega\subseteq \R^3$ open. Furthermore, we shall assume in most of the statements in this section that $\Omega$ satisfies the \textbf{electric compactness property}; that is, 
\[
    \dom(\dive)\cap \dom(\curlc)\hookrightarrow \hookrightarrow L_2(\Omega)^3.
\]By a contradiction argument, we infer that $\ran(\curlc)$ and -- by Proposition \ref{prop:crt} -- $\ran(\curl)$ are closed; see Proposition \ref{prop:compclos} and Example \ref{ex:rangeclosed}(b) for the details. Furthermore, it follows that $\mathcal{H}_D(\Omega)=\ker(\dive)\cap \ker(\curlc)$ is finite-dimensional. Moreover, we remark here that any $\Omega$ satisfying the electric compactness property is necessarily bounded. Hence, by Poincar\'e's inequality,  $\ran(\gradc)$ is closed. Consequently, $\Omega$ is range closed. We remark that the Maxwell compactness property is stronger than the electric compactness property. Thus, by the Picard--Weber--Weck selection theorem, for instance bounded weak Lipschitz domains $\Omega$ satisfy the electric compactness property.

Next, we provide an explicit reformulation of Theorem \ref{thm:finite}. As short hand, we will use 
\begin{equation}\label{eq:shtb}
   \tau_{\textnormal{bH}}\coloneqq \tau_{\textnormal{b}}(\ran(\mathring{\grad}),\ran(\mathring{\grad})^{\bot}),
\end{equation}
where we introduced the notion of bounded nonlocal $H$-convergence $\tau_{\textnormal{b}}$ just after Theorem \ref{thm:sep}. Quickly recall that $\varepsilon \in \cB(L_2(\Omega)^3)$ is called admissible, if
\begin{enumerate}
 \item[{(a1)}] $\varepsilon$ is invertible,\vspace*{0.1cm}
 \item[{(a2)}] $\iota^*\varepsilon\iota$ is invertible, where $\iota\colon \ran(\gradc)\hookrightarrow L_2(\Omega)^3$, and\vspace*{0.1cm}
  \item[{(a3)}] $\kappa^*\varepsilon^{-1}\kappa$ is invertible, where $\kappa\colon \ran(\curl)\hookrightarrow L_2(\Omega)^3$.\vspace*{0.1cm}
\end{enumerate}
\begin{theorem}\label{thm:charnonlb} Let $\Omega\subseteq \R^3$ open satisfying the electric compactness property. Then 
\[
 \tau_{\textnormal{bH}} =
\tau_{\textnormal{b}}(\ran(\curl)^\bot,\ran(\curl))
\]
on admissible operators in $\cB(L_2(\Omega)^3)$.
In particular, for admissible sequences $(\varepsilon_n)_n$ and an admissible $\varepsilon$ the following conditions are equivalent:
\begin{enumerate}
\item[(i)] for all $f\in H^{-1}(\Omega)$ and $u_n\in H_0^1(\Omega)$ satisfying
\[
    -\dive\varepsilon_n\gradc u_n= f,
\]
we obtain $u_n\rightharpoonup u$ in $H_0^1(\Omega)$ and $\varepsilon_{n,00}^{-1}\varepsilon_{n,01}$, $\varepsilon_{n,10}\varepsilon_{n,00}^{-1}$, and $\varepsilon_{n,11}-\varepsilon_{n,10}\varepsilon_{n,00}^{-1}\varepsilon_{n,01}$ respectively converging in the weak operator topology to $\varepsilon_{00}^{-1}\varepsilon_{01}$, $\varepsilon_{10}\varepsilon_{00}^{-1}$, and $\varepsilon_{11}-\varepsilon_{10}\varepsilon_{00}^{-1}\varepsilon_{01}$, where we used the decomposition $L_2(\Omega)^3=\ran(\gradc)\oplus \ker(\dive)$.
\item[(ii)]for all $g\in H^{-1}(\curlr)$ and $v_n\in \dom(\curlr)$ satisfying
\[
  \curlrd\varepsilon_n^{-1}\curlr v_n= g,
\]
we obtain $v_n\rightharpoonup v$ in $H^1(\curlr)$ and $\varepsilon_{n,00}^{-1}$, $\varepsilon_{n,00}^{-1}\varepsilon_{n,01}$, and $\varepsilon_{n,10}\varepsilon_{n,00}^{-1}$, respectively converging in the weak operator topology to $\varepsilon_{00}^{-1}$, $\varepsilon_{00}^{-1}\varepsilon_{01}$, and $\varepsilon_{10}\varepsilon_{00}^{-1}$, where we used the decomposition $L_2(\Omega)^3= \ker(\curlc)\oplus \ran(\curl)$.
\end{enumerate}
\end{theorem}
\begin{proof}The assertion follows from Theorem \ref{thm:finite} together with the consequence of the electric compactness property in that $\mathcal{H}_{\textnormal{D}}(\Omega) = \ran(\gradc)^{\bot}\cap \ran(\curlr)^\bot = \ker(\dive)\cap \ker(\curlc)$ is finite-dimensional.
For the equivalence stated note that this is just a reformulation of the equality of the topologies.
\end{proof}

\begin{remark}\label{rem:correcttopology0}
If $\Omega$ is range closed (not necessarily having finite-dimensional $\cH_D(\Omega)$), then the convergence described in (i) of Theorem \ref{thm:charnonlb} is equivalent to convergence in $\tau_{\textnormal{bH}}$.
\end{remark}

Particularly useful for the next section is the following consequence of the previous theorem.

\begin{corollary}\label{cor:nonlquadr} Let $\Omega\subseteq \R^3$  open satisfying the electric compactness property. Assume that a sequence of admissible operators $(\varepsilon_n)_n$ converges to some admissible $\varepsilon$ in  $\tau_{\textnormal{bH}}$. Then the following conditions are satisfied:
\begin{enumerate}
\item[(a)] Let $(f_n)_n$ in $H^{-1}(\Omega)$, $f\in H^{-1}(\Omega)$. Assume for $n\in \N$ that $u_n\in H_0^1(\Omega)$ satisfies
\[
    -\dive\varepsilon_n\mathring{\grad} u_n= f_n
\]
If $f_n\to f$ in $H^{-1}(\Omega)$, we obtain $u_n\rightharpoonup u$ weakly in $H_0^1(\Omega)$, and $\varepsilon_n\gradc u_n\rightharpoonup \varepsilon \gradc u$ weakly in $L_2(\Omega)^3$.
\item[(b)] Let $(g_n)_n$ in $H^{-1}(\curlr)$, $g\in H^{-1}(\curlr)$. Assume for $n\in \N$ that $v_n\in \dom(\curl_{\rdd})$ satisfies
\[
    \curlrd\varepsilon_n^{-1}\curlr v_n= g_n.
\]
Then $v_n\rightharpoonup v$ in $H^1(\curlr)$, $\varepsilon_n^{-1}\curl v_n \rightharpoonup \varepsilon^{-1}\curl v$ in $L_2(\Omega)^3$.
\item[(c)] $\gradc(\dive\varepsilon_n\gradc)^{-1}\dive\varepsilon_n\to \gradc(\dive\varepsilon\gradc)^{-1}\dive\varepsilon$ in the weak operator topology.
\item[(d)] $\varepsilon_n(1-\gradc(\dive\varepsilon_n\gradc)^{-1}\dive\varepsilon_n)\to
\varepsilon(1-\gradc(\dive\varepsilon\gradc)^{-1}\dive\varepsilon)$ in the weak operator topology.
\end{enumerate}
\end{corollary}
\begin{proof} For (a) and (b), in view of Theorem \ref{thm:charnonlb}, it suffices to show convergence of the respective fluxes (Note that the weak convergence of the solutions $(u_n)_n$ and $(v_n)_n$ follow from the reformulation in Theorem \ref{thm:TW-Mana}  (b) and the assumed convergence of the operators involved in the weak operator topology). By the symmetry of the argument (see Remark \ref{rem:inverse}), it moreover suffices to consider (a). For this, note that 
\[
u_n = -(\dive\varepsilon_n\mathring{\grad})^{-1}f = -(\iota_{0}^*\mathring{\grad})^{-1}\varepsilon_{n,00}^{-1}(\dive\iota_0)^{-1}f
\]and so
\begin{align*}
   \varepsilon_n\mathring{\grad} u_n & =
   -   \varepsilon_n\iota_0\iota_0^*\mathring{\grad}(\iota_{0}^*\mathring{\grad})^{-1}\varepsilon_{n,00}^{-1}(\dive\iota_0)^{-1}f \\
   & =
   - \begin{pmatrix}\iota_0 & \iota_1 \end{pmatrix} \begin{pmatrix}\iota_0^* \\ \iota_1^* \end{pmatrix}  \varepsilon_n\begin{pmatrix}\iota_0 & \iota_1 \end{pmatrix} \begin{pmatrix}\iota_0^* \\ \iota_1^* \end{pmatrix}\iota_0\iota_0^*\mathring{\grad}(\iota_{0}^*\mathring{\grad})^{-1}\varepsilon_{n,00}^{-1}(\dive\iota_0)^{-1}f 
   \\ &  =
   - \begin{pmatrix}\iota_0 & \iota_1 \end{pmatrix} \begin{pmatrix}\varepsilon_{n,00} & \varepsilon_{n,01} \\ \varepsilon_{n,10} & \varepsilon_{n,11} \end{pmatrix}  \begin{pmatrix}\iota_0^*\iota_0\varepsilon_{n,00}^{-1}(\dive\iota_0)^{-1}f  \\ 0 \end{pmatrix} \\
    & =- \begin{pmatrix}\iota_0 & \iota_1 \end{pmatrix} \begin{pmatrix}\varepsilon_{n,00} & \varepsilon_{n,01} \\ \varepsilon_{n,10} & \varepsilon_{n,11} \end{pmatrix}  \begin{pmatrix}\iota_0^*\iota_0\varepsilon_{n,00}^{-1}(\dive\iota_0)^{-1}f  \\ 0 \end{pmatrix} = \iota_1\varepsilon_{n,10}\varepsilon_{n,00}^{-1}(\dive\iota_0)^{-1}f,
\end{align*}
so the convergence of the fluxes follows from the convergence of $\varepsilon_{n,10}\varepsilon_{n,00}^{-1}$ in the weak operator topology.

(c) We compute using the decomposition $L_2(\Omega)^3= \ran(\mathring{\grad})\oplus \ker(\dive)$,
\begin{align*}
 \mathring{\grad}(\dive\varepsilon_n\mathring{\grad})^{-1}\dive\varepsilon_n & =
\iota_0\iota_0^* \mathring{\grad} (\iota_{0}^*\mathring{\grad})^{-1}\varepsilon_{n,00}^{-1}(\dive\iota_0)^{-1}\dive\iota_{0}\iota_0^*\varepsilon_n \\
& = \iota_0\varepsilon_{n,00}^{-1}\iota_0^*\begin{pmatrix}\iota_0 & \iota_1 \end{pmatrix}  \begin{pmatrix}\varepsilon_{n,00} & \varepsilon_{n,01} \\ \varepsilon_{n,10} & \varepsilon_{n,11} \end{pmatrix} \begin{pmatrix}\iota_0^* \\ \iota_1^* \end{pmatrix}\\
& = \begin{pmatrix}\iota_0\varepsilon_{n,00}^{-1} & 0 \end{pmatrix}  \begin{pmatrix}\varepsilon_{n,00}\iota_0^*+ \varepsilon_{n,01}\iota_1^* \\ \varepsilon_{n,10}\iota_0^* +\varepsilon_{n,11}\iota_1^*  \end{pmatrix} \\
& =\iota_0\varepsilon_{n,00}^{-1}(\varepsilon_{n,00}\iota_0^*+ \varepsilon_{n,01}\iota_1^* )\\&=\iota_0\varepsilon_{n,00}^{-1}\varepsilon_{n,00}\iota_0^* +\iota_0\varepsilon_{n,00}^{-1}\varepsilon_{n,01}\iota_1^*.
\end{align*}
Hence, the statement follows from the convergence statement in (i) of Theorem \ref{thm:charnonlb}.

(d) With the same decomposition as in (c), we compute
\begin{align*}
& \varepsilon_n(1-\gradc(\dive \varepsilon_n\gradc)^{-1}\dive\varepsilon_n) \\ &= \varepsilon_n(1-\iota_0\iota_0^* +\iota_0\varepsilon_{n,00}^{-1}\varepsilon_{n,01}\iota_1^*) \\
& = \begin{pmatrix}\iota_0 & \iota_1 \end{pmatrix}  \begin{pmatrix}\varepsilon_{n,00} & \varepsilon_{n,01} \\ \varepsilon_{n,10} & \varepsilon_{n,11} \end{pmatrix} \begin{pmatrix}\iota_0^* \\ \iota_1^* \end{pmatrix}(\iota_1\iota_1^* -\iota_0\varepsilon_{n,00}^{-1}\varepsilon_{n,01}\iota_1^*) \\
& = \begin{pmatrix}\iota_0 & \iota_1 \end{pmatrix}  \begin{pmatrix}\varepsilon_{n,00} & \varepsilon_{n,01} \\ \varepsilon_{n,10} & \varepsilon_{n,11} \end{pmatrix} \begin{pmatrix}-\varepsilon_{n,00}^{-1}\varepsilon_{n,01} \\ 1 \end{pmatrix}\iota_1^* \\
& = \begin{pmatrix}\iota_0 & \iota_1 \end{pmatrix}  \begin{pmatrix}0\\ \varepsilon_{n,11}  -\varepsilon_{n,10} \varepsilon_{n,00}^{-1}\varepsilon_{n,01}  \end{pmatrix}\iota_1^* =\iota_1(\varepsilon_{n,11}  -\varepsilon_{n,10} \varepsilon_{n,00}^{-1}\varepsilon_{n,01})\iota_1^*.
\end{align*}The latter operator converges in the weak operator topology by Theorem \ref{thm:charnonlb}(i). 
\end{proof}
\begin{remark} Note that (a), (c) and (d) are true even in the case where $\mathcal{H}_D(\Omega)$ is infinite-dimensional; that is, range closed $\Omega$ is sufficient.
\end{remark}

We now turn to the discussion of the homogenisation problem for electrostatic problems as discussed in Section \ref{sec:estatic}. We particularly refer to Remark \ref{rem:wpmod} for the equivalent formulations of \eqref{prb:class} and \eqref{prb:opt} used in the following theorem.

\begin{theorem}\label{thm:homelectrostatics} Let $\Omega\subseteq\R^3$ be open satisfying the electric compactness property. Let $(\varepsilon_n)_n$ in $\cB(L_2(\Omega)^3)$ and $\varepsilon\in \cB(L_2(\Omega)^3)$. Assume $\varepsilon$ and, for all $n\in \N$, $\varepsilon_n$ are admissible. Furthermore, assume $\varepsilon_n\to \varepsilon$ in $\tau_{\textnormal{bH}}$.

Let $f\in H^{-1}(\Omega)$, $g\in H^{-1}(\curl_{\rdd})$, $x\in \mathcal{H}_{D}(\Omega)$ and $E_n\in L_2(\Omega)^3$ satisfying
\[
   \begin{cases} \dive\varepsilon_nE_n = f&\\
   \curl E_n = g &\\
   \pi_D\pi_{\varepsilon_n}E_n = x.& \end{cases}
\]
Then $E_n \to E$ and $\varepsilon_n E_n \to \varepsilon E$ weakly in $L_2(\Omega)^3$ and $E$ satisfies
\[
 \begin{cases} \dive\varepsilon E = f&\\
   \curl E  = g &\\
    \pi_D\pi_{\varepsilon}E = x.& \end{cases}
\]
\end{theorem}
\begin{proof}
We use the solution formulas provided in Theorem \ref{thm:solform} with 
\begin{align*}
x_{\varepsilon_n} &= (1-\gradc(\dive \varepsilon_n \gradc)^{-1}\dive \varepsilon_n) x, \\
x^{\varepsilon_n} &= \varepsilon_n(1-\gradc(\dive \varepsilon_n \gradc)^{-1}\dive \varepsilon_n) x
\end{align*}
 (see also Remark \ref{rem:wpmod}) to obtain 
\begin{multline*}
  E_n = \gradc( \dive\varepsilon_n\gradc )^{-1}f \\+\varepsilon_n^{-1}\curlr (\curlrd\varepsilon_n^{-1}\curlr)^{-1} g +(1-\gradc(\dive \varepsilon_n \gradc)^{-1}\dive \varepsilon_n) x\end{multline*}
  and
  \begin{multline*}
 H_n=\varepsilon_n E_n = \varepsilon_n\gradc( \dive\varepsilon_n\gradc )^{-1}f \\+\curlr (\curlrd\varepsilon_n^{-1}\curlr)^{-1} g +\varepsilon_n(1-\gradc(\dive \varepsilon_n \gradc)^{-1}\dive \varepsilon_n) x.\end{multline*}
Here, using Corollary \ref{cor:nonlquadr}, we may now let $n\to\infty$ to obtain the desired result.
\end{proof}

\begin{remark}\label{rem:movingrhs} An analogous statement to Theorem \ref{thm:homelectrostatics} holds, if we replace $f$, $g$ and $x$ by (strongly) convergent sequences $(f_n)_n$, $(g_n)_n$ and $(x_n)_n$. 
\end{remark}

\section{Convergence of energy and the div-curl Lemma}\label{sec:energy}

Similarly to the previous section, we shall assume that $\Omega\subseteq \R^3$ satisfies the electric compactness property. We also recall the abbreviation 
\[
   \tau_{\textnormal{bH}}  =\tau_{\textnormal{b}}(\ran(\mathring{\grad}),\ran(\mathring{\grad})^{\bot}).
\]
At the end of this section, we provide a proof for Theorem \ref{thm:hometstatic}.

 In order to obtain convergence of the energy $\langle \varepsilon_n E_n,E_n\rangle$ for $E_n$ solving the system for electrostatics (see \cite[p 64]{E93}), we need the following concrete version of the abstract div-curl lemma from \cite[Theorem 2.6]{W17_DCL}. 
\begin{theorem}\label{thm:dcl} Let $\Omega\subseteq \R^3$ open satisfying the electric compactness property. Assume that $(r_n)_n,(q_n)_n$ are weakly convergent sequences in $L_2(\Omega)^3$ with respective limits $r$ and $q$. If both $\{\dive r_n; n\in \N\}\subseteq H^{-1}(\Omega)$ and $\{\curld q_n; n\in \N\}\subseteq H^{-1}(\curl)$ are relatively compact, then
\[
    \lim_{n\to\infty}\langle r_n,q_n\rangle_{L_2(\Omega)^3} = \langle r,q\rangle_{L_2(\Omega)^3}.
\]
\end{theorem}
\begin{proof}
The claim follows from \cite[Theorem 2.6]{W17_DCL} for the special case $A_0 =\gradc$ and $A_1=\curlc$ (where we used the notation there). The assumptions in \cite[Theorem 2.6]{W17_DCL} are that $\ran(A_0)$ and $\ran(A_1)$ are closed and that $\ker(A_0^*)\cap\ker(A_1)$ is finite-dimensional. These assumptions are -- as it is argued just after the definition of the electric compactness property -- consequences of this compactness property.
\end{proof}

We state the main result of this section next.

\begin{theorem}[Convergence of the energy]\label{thm:conenergy}
Let $(\varepsilon_n)_n$ be admissible and $\tau_{\textnormal{bH}}$-convergent to some admissible $\varepsilon$. If $f_n\to f$ and $g_n\to g$ strongly in $H^{-1}(\Omega)$ and $H^{-1}(\curl_{\textnormal{red}})$, $x_n\to x$  in $\mathcal{H}_{D}(\Omega)$. Then, given $E_n$ satisfies
\[
  \dive \varepsilon_n E_n = f_n \quad \curlc E_n = g_n, \pi_{D}\pi_{\varepsilon_n}E_n =x_n
\]
we obtain that both $E_n\to E$ and $\varepsilon_nE_n\to \varepsilon E$ weakly in $L_2(\Omega)^3$, where $E$ satisfies
\[
\dive \varepsilon E = f \quad \curlc E  = g, \pi_{D}\pi_{\varepsilon} =x.
\]
Moreover,  
\[
   \langle \varepsilon_nE_n,E_n\rangle  \to    \langle \varepsilon E,E\rangle 
\]
\end{theorem}
\begin{proof}
 The weak convergence of $E_n$ and $\varepsilon_nE_n$ to $E$ and $\varepsilon E$ follows from Theorem \ref{thm:homelectrostatics} in conjunction with Remark \ref{rem:movingrhs}.

The convergence of the scalar products, $\langle \varepsilon_nE_n,E_n\rangle  \to    \langle \varepsilon E,E\rangle $, follows from Theorem \ref{thm:dcl}. Indeed, for $q_n =E_n$ and $r_n=\varepsilon_n E_n$ the assumptions of Theorem \ref{thm:dcl} are satisfied proving the assertion.
\end{proof}

\begin{proof}[Proof of Theorem \ref{thm:hometstatic}] This is a direct consequence of Theorem \ref{thm:homelectrostatics} together with Theorem \ref{thm:conenergy}.
\end{proof}

\section{A compactness criterion with moving coefficients}\label{sec:compact}

In this section, we present a compactness criterion and provide a proof of Theorem \ref{thm:comprough}. For this, let $\Omega\subseteq \R^3$ be open. Consider a sequence of admissible coefficients $(\varepsilon_n)_n$ in $\cB(L_2(\Omega)^3)$. When is it true that a bounded $L_2$-sequence of vector fields $(E_n)_n$ satisfying the conditions
\[
   \dive\varepsilon_nE_n, \curlc E_n \text{ bounded in $L_2$}
\]
actually contains an $L_2$-convergent subsequence? If $\varepsilon_n=\varepsilon_0$ for some selfadjoint, strictly positive definite $\varepsilon_0$, then standard Hilbert space techniques confirm that 
\[
   \dom(\dive\varepsilon_0)\cap \dom(\curlc)\hookrightarrow \hookrightarrow L_2(\Omega)^3,
\]if $\Omega$ satisfies the electric compactness property, see, e.g.,~\cite{Pauly2021}. Thus, the compactness result hinges on $\Omega$, only. In the case of `moving' coefficients, that is, non-constant $(\varepsilon_n)_n$, some additional assumptions on the limit behaviour are needed. The corresponding result reads as follows, which by the Picard--Weber--Weck selection theorem is more general than Theorem \ref{thm:comprough} from the introduction.

\begin{theorem}[Helga's Theorem\footnote{This theorem is dedicated to my daughter Helga Mariella Ostmann. In fact, when I was carrying her through the night to make her sleep, I was wondering about general results of mine; this was the moment, when it occurred to me that a combination of nonlocal $H$-convergence and weak operator convergence implies strong convergence for the electric field.}]\label{thm:compact} Let $\Omega\subseteq \R^3$ open satisfying the electric compactness property. Let  $(\varepsilon_n)_{n\geq1}$ be bounded sequence in $\cB(L_2(\Omega)^3)$, $\varepsilon_0\in \cB(L_2(\Omega)^3)$. Assume there exists $c>0$ such that
\[
   \Re \varepsilon_n \geq c\quad (n\in \N_{\geq 1}).
\] and 
assume that for all $E\in L_2(\Omega)^3$, $(\dive \varepsilon_n E)_n$ is relatively compact in $H^{-1}(\Omega)$. 

Assume $\varepsilon_n\to \varepsilon_0$ as $n\to\infty$ in both the weak operator topology and $\tau_{\textnormal{bH}}$. 

Let $(E_n)_n$ be bounded in $L_2(\Omega)^3$ such that
\[
     ( \dive \varepsilon_n E_n)_n\text{ and } (\curlc E_n)_n
\]are bounded in $L_2(\Omega)$ and $L_2(\Omega)^3$. 

Then there exists an $L_2(\Omega)^3$-strongly convergent subsequence of $(E_n)_n$.
\end{theorem}
\begin{proof}[Proof of Theorem \ref{thm:compact} and, thus, of Theorem \ref{thm:comprough}] Without loss of generality, we may assume that $E_n\to E$ weakly in $L_2(\Omega)^3$ for some $E\in L_2(\Omega)^3$. By Remark \ref{rem:posdefadm}, $\varepsilon_n$ is admissible for all $n\in \N$. Since $\varepsilon_n\to \varepsilon_0$ in the weak operator topology, we obtain $\Re \varepsilon_0\geq c$ and, thus, $\varepsilon_0$ is, too, admissible. We define $\tilde{f}_n \coloneqq  \dive \varepsilon_n E_n$ and $\tilde{g}_n\coloneqq \mathring{\curl} E_n$. Next, we put
\[
   f_n \coloneqq \tilde{f}_n-\dive \varepsilon_n E; \text{ and }   g_n \coloneqq \tilde{g}_n-\mathring{\curl} E.
\]Then, by assumption for another subsequence, $f_n\to f\coloneqq\tilde{f}-\dive \varepsilon E$ and $g_n\to g\coloneqq\tilde{g}-\mathring{\curl} E$ strongly in $H^{-1}(\Omega)$ and $H^{-1}(\curl_{\text{red}})$, respectively. Hence, as $\varepsilon_n\to\varepsilon_0$ in $\tau_{\textnormal{bH}}$, we obtain by Theorem \ref{thm:conenergy}, $E_n\to E$ and $\varepsilon_nE_n\to \varepsilon E$  weakly in $L_2(\Omega)^3$. Moreover, by Theorem \ref{thm:conenergy} again, we deduce
\begin{align*}
 c\langle E_n-E,E_n -E\rangle &\leq \Re \langle E_n-E,\varepsilon_n(E_n-E)\rangle \to 0.
\end{align*}
The claim follows.
\end{proof}

\begin{corollary} Let $\Omega\subseteq\R^3$ open satisfying the electric compactness property. Let $\varepsilon\in \cB(L_2(\Omega)^3)$ with
\[
   \Re \varepsilon\geq c
\]for some $c>0$. Then
\[
   \dom(\dive\varepsilon)\cap\dom(\curlc)\hookrightarrow\hookrightarrow L_2(\Omega)^3.
\]
\end{corollary}
\begin{proof}
The assertion follows trivially from Theorem \ref{thm:compact} for the constant sequence $\varepsilon_n=\varepsilon$ with limit $\varepsilon_0=\varepsilon$.
\end{proof}

\begin{remark}\label{rem:sot}
If $\varepsilon_n\to \varepsilon_0$ in the strong operator topology additionally satisfies the divergence condition, then the assumptions in Theorem \ref{thm:compact} are met.
\end{remark}

The rest of this section is devoted to discuss an example. More precisely, we shall establish the following result.

\begin{theorem}\label{thm:examp} Let $\rho\in L_\infty(\R^3)\cap L_1(\R^3)$ and assume that $\rho_n\coloneqq \rho(n\cdot)\to \tilde \rho$ in $\sigma(L_\infty(\R^3),L_1(\R^3))$ for some $\tilde\rho\in L_\infty(\R^3)$. If $\|\rho\|_{L_1}<1$, then $(\varepsilon_n)_{n\geq1}$ in $\cB(L_2(\Omega))$ given by
\[
     \varepsilon_n\phi \coloneqq (1+\rho_n*)\phi \coloneqq \phi+ (x\mapsto \int_\Omega \rho_n(x-y)\phi(y)\dd y)
\]
satisfies the assumptions of Theorem \ref{thm:compact} with $\varepsilon_0=1+\tilde\rho*$.
\end{theorem}

We prove Theorem \ref{thm:examp} in several lemmas.

\begin{lemma}\label{lem:unifbdd} Let $\rho\in L_1(\R^3), E\in L_2(\Omega)^3$; $\rho_n\coloneqq \rho(n\cdot)$. Then $\{((1+\rho_n*)E); n\in \N, n\geq1\}\subseteq L_2(\Omega)^3$ is relatively compact.
\end{lemma}
\begin{proof}
We prove the claim in two steps. 

First step: Assume additionally that $\rho\in W^1_1(\R^3)$. In this case, for $j\in \{1,2,3\}$, we note $\partial_j \rho_n = n(\partial_j\rho)(n\cdot)$ and 
 \[
    \| \partial_j\rho_n\|_{L_1(\Omega)} = \int_{\Omega} n |\partial_j \rho(n x)| \dd x = \int_{n\Omega} |\partial_j \rho(x)|\dd x \leq \|\partial_j \rho\|_{L_1(\R^3)}. 
 \]
 Hence, $(\partial_j \rho_n*E)_{n\in \N, j\in \{1,2,3\}}$ is uniformly bounded in $L_2(\Omega)^3$. In particular, $(\rho_n*E)_n$ is uniformly bounded in $H^1(\Omega)^3$ and, by Rellich's selection theorem, $\{\rho_n*E; n\in \N,n\geq 1\}$ is relatively compact in $L_2(\Omega)^3$, proving the assertion for the first step.
  
 Second step: We show that $(\rho_n*E)_n$ is totally bounded in $L_2(\Omega)^3$. For this, let $\varepsilon>0$. Then we find $\eta\in C_c^\infty(\Omega)$ such that $\|\rho-\eta*\rho\|_{L_1}\leq \varepsilon$. By the first step, we find a finite set $N\subseteq \N$ such that
 \[
     \bigcup_{n\in N} B( (\eta*\rho)_n*E,\varepsilon) \supseteq \{(\eta*\rho)_n*E; n\in\N\}.
 \]
  Thus, with Young's inequality again, for $n\geq 1$,
  \[
     \|\rho_n*E- (\eta*\rho)_n*E\|_{L_2}\leq \frac1n \|\rho-\eta*\rho\|_{L_1}\|E\|_{L_2}\leq \varepsilon \|E\|_{L_2},
  \]  the assertion follows.
\end{proof}
\begin{lemma}\label{lem:wot} Let $(\kappa_n)_n$ bounded in $L_\infty(\R^3)\cap L_1(\R^3)$ and  $\kappa\in L_\infty(\R^n)$. Assume that $\kappa_n\to \kappa$ in $\sigma(L_\infty(\R^3),L_1(\R^3))$. Then $\kappa_n*\to \kappa*$ in the weak operator topology of $\cB(L_2(\Omega)^3)$. If, additionally, $\kappa_n =\rho(n\cdot)$ for some $\rho\in L_1(\R^3)$, then $\kappa_n*\to \kappa*$ in the strong operator topology.
\end{lemma}
\begin{proof}
  By Young's inequality, $(\kappa_n*)_n$ considered as an operator sequence in $\cB(L_2(\R^3))$ is uniformly bounded. First we show convergence in the weak operator topology. For this, by the observed boundedness, it suffices to prove convergence of expressions of the type $\langle \kappa_n*\psi,\phi\rangle$ for $\psi,\phi\in C_c^\infty(\Omega)$, only. The latter, however, is a consequence of Fubini's and the dominated convergence theorem. Moreover, in the particular case $\kappa_n=\rho(n\cdot)$ by Lemma \ref{lem:unifbdd}, $(\kappa_n*E)_n$ relatively compact in $L_2(\Omega)^3$. This together with weak convergence proves strong operator topology convergence.
\end{proof}

\begin{proof}[Proof of Theorem \ref{thm:examp}] By Lemma \ref{lem:unifbdd}, it follows that $((1+\rho_n*)E)_n$ is relatively compact in $L_2(\Omega)^3$. Hence, by the continuity of $\dive \colon L_2(\Omega)^3\to H^{-1}(\Omega)$, it follows that $(\dive (1+\rho_n*)E)_n$ is relatively compact in $H^{-1}(\Omega)$. Moroever, we obtain
\[
     \Re (1+\rho_n*)\geq 1-\sup_{n}\|\rho_n\|_{L_1}>0.
\]
Finally, as the strong operator topology of bounded sequences is stronger both than the weak operator topology and $\tau_{\textnormal{bH}}$, the assertion follows (see also Remark \ref{rem:sot}).
\end{proof}

Being formulated for potentially topologically nontrivial domains, convergence in $\tau_{\textnormal{bH}}$ generalises the concept of nonlocal $H$-convergence for bounded sequences as introduced in \cite{W18_NHC}. Since this generalisation consists in the consideration of other domains, the theorem stating the equivalence of local $H$-convergence and nonlocal $H$-convergence as presented in \cite[Theorem 5.11]{W18_NHC} or Theorem \ref{thm:locnonloc0} below does not apply anymore. However, in the next section, we show that local $H$-convergence is equivalent to convergence in $\tau_{\textnormal{bH}}$ for multiplication operators confirming the `correctness' of our generalisation, see also Section \ref{sec:nHc}.

\section{Relationship to local $H$-convergence}\label{sec:locnonlocnontriv}

Finally, we position convergence in $\tau_{\textnormal{bH}}$ in the context of local $H$-convergence; that is, we provide a proof of Theorem \ref{thm:locn}. As before, we let throughout $\Omega\subseteq\R^3$ open satisfying the electric compactness property, i.e., 
\[
    \dom(\dive)\cap \dom(\curlc)\hookrightarrow \hookrightarrow L_2(\Omega)^3.
\] and
\[
   \tau_{\textnormal{bH}}  =\tau_{\textnormal{b}}(\ran(\mathring{\grad}),\ran(\mathring{\grad})^{\bot}).
\]
Moreover, we define for $0<\alpha<\beta$ 
\[
   M(\alpha,\beta;\Omega)\coloneqq \{ a\in L_\infty(\Omega)^{3\times 3}; \Re a\geq \alpha,\Re a^{-1}\geq 1/\beta \}.
\]
Recall that a sequence $(a_n)_n$ in $M(\alpha,\beta;\Omega)$ \textbf{locally $H$-converges} or \textbf{converges in the local $H$-sense} to some $a\in M(\alpha,\beta;\Omega)$, if for all $f\in H^{-1}(\Omega)$ and $u_n\in H_0^1(\Omega)$ satisfying
\[
   -\dive a_n\gradc u_n = f,
\]one obtains $u_n\to u$ weakly and $ a_n\gradc u_n \to a\gradc u$ weakly in $H_0^1(\Omega)$ and $L_2(\Omega)^3$, respectively.  The main result of this section reads as follows.

\begin{theorem}\label{thm:locstrongnonloc} Let $(\varepsilon_n)_n$ in $M(\alpha,\beta;\Omega)$ and $\varepsilon\in M(\alpha,\beta;\Omega)$. Assume that $\varepsilon_n\to \varepsilon$ in the local $H$-sense. Then $\varepsilon_n\to \varepsilon$ in $\tau_{\textnormal{bH}}$.
\end{theorem}

The proof of this theorem needs some preparations. First, we recall a characterisation of nonlocal $H$-convergence for topologically trivial domains from \cite{W18_NHC}. Even though the proof in \cite{W18_NHC} is provided in a slightly less general setting for the class of operator sequences considered, the proof for admissible operators is (almost) literally the same. The required modifications are obvious. 

\begin{theorem}[{{\cite[Theorem 6.2; see also Theorem 6.5]{W18_NHC}}}]\label{thm:DCLcht} Let $\Omega_{\textnormal{t}}\subseteq \mathbb{R}^3$ be open satisfying the electric compactness property and $\cH_D(\Omega_\textnormal{t})=\{0\}$.  Let $(a_n)_n$ be bounded in $\cB(L_2(\Omega_\textnormal{t})^3)$ and admissible, $a\in \cB(L_2(\Omega_\textnormal{t})^3)$ admissible. Then the following statements are equivalent:
\begin{enumerate}
  \item[(i)] $(a_n)_n$ converges to $a$ in $\tau(\ran(\gradc,{\Omega_{\textnormal{t}}}), \ran(\curl,{\Omega_{\textnormal{t}}}))$
      \item[(ii)] for all $(q_n)_n$ in $L_2(\Omega_{\textnormal{t}})^3$ weakly convergent to some $q$ in $L_2(\Omega_{\textnormal{t}})^3$ and $\kappa\colon \mathbb{N}\to\mathbb{N}$ strictly monotone we have: Given the conditions
    \begin{enumerate}
      \item[(a)] $(\dive (a_{\kappa(n)} q_n))_n$ is relatively compact in $H^{-1}(\Omega_{\textnormal{t}})$,
      \item[(b)] $(\interior{\curl} (q_n))_n$ is relatively compact in $H^{-1}(\curlr,\Omega_{\textnormal{t}})$,
    \end{enumerate}
    then $a_{\kappa(n)}q_n\rightharpoonup aq$ as $n\to\infty$.
\end{enumerate}
\end{theorem}

We recall the equivalence theorem of local and nonlocal $H$-convergence from \cite{W18_NHC} for topologically trivial domains. 

\begin{theorem}[{{\cite[Theorem 5.11]{W18_NHC}}}]\label{thm:locnonloc0} Let $\Omega_{\textnormal{t}}\subseteq \mathbb{R}^3$ be open satisfying the electric compactness property and $\cH_D(\Omega_\textnormal{t})=\{0\}$.  Let $(a_n)_n$ be a sequence in $M(\alpha,\beta;\Omega_{\textnormal{t}})$, $a\in M(\alpha,\beta;\Omega_{\textnormal{t}})$.
 Then the following statements are equivalent:
\begin{enumerate}
  \item[(i)] $(a_n)_n$ converges to $a$ in $\tau(\ran(\gradc,{\Omega_{\textnormal{t}}}), \ran(\curl,{\Omega_{\textnormal{t}}}))$
      \item[(ii)]  $(a_n)_n$ locally $H$-converges to $a$.
\end{enumerate}
\end{theorem}

Next, we establish the main result of this section.

\begin{proof}[Proof of Theorem \ref{thm:locstrongnonloc}]
We will show the property described in Theorem \ref{thm:charnonlb} (ii). For the operators in question, we use the orthogonal decomposition induced by $\ran(\curl)^\bot \oplus \ran(\curl)$; the induced operators are denoted by indices as in Theorem \ref{thm:charnonlb} (ii). 
Before we turn to the $\curl$-problem, we show that $\varepsilon_{n,00}^{-1}$ converges in the weak operator topology to $\varepsilon_{00}^{-1}$.
For this, one notes that the operator
\[
   G\coloneqq \begin{pmatrix} \gradc_{\Omega} & \iota_D \end{pmatrix} \colon H^1_0(\Omega) \oplus \cH_D(\Omega) \subseteq L_2(\Omega) \oplus \cH_D(\Omega) \to L_2(\Omega)^3, (u,x)\mapsto \gradc u +x
\]is densely defined and closed. Moreover, 
\[
\ran(G)=\ran(\gradc,\Omega)+\cH_D(\Omega) (=\ran(\curl)^\bot)
\]
is closed as the sum is orthogonal and both $\ran(\gradc,\Omega)$ and $\cH_D(\Omega)$ are closed. Next, $G$ is one-to-one. Consequently, for $f\in L_2(\Omega)$ and $y\in \mathcal{H}_{D}(\Omega)$, we find uniquely determined $u_n\in H_0^1(\Omega)$ and $x_n\in \cH_D(\Omega)$ such that
\[
   G^* \varepsilon_n G\begin{pmatrix} u_n \\ x_n\end{pmatrix} = \begin{pmatrix} f \\ y \end{pmatrix}.
\] It follows that $U_n\coloneqq\begin{pmatrix} u_n \\ x_n\end{pmatrix}$ forms a bounded sequence in $\dom(G)$. Thus, there exists a weakly convergent subsequence $(U_{\kappa(n)})_n$ with limit $U$. We need to uniquely identify $U$. For this, consider $q_n\coloneqq G U_{\kappa(n)}$ and let $B\subseteq \Omega$ be an open ball with $\overline{B}\subseteq \Omega$. Since local $H$-convergence on $\Omega$ implies the same on $B$, we infer (that the restriction to $L_2(B)^3$ of) $(\varepsilon_n)_n$ locally $H$-converges to $\varepsilon$. As $B^{\textnormal{c}}$ is connected, $\cH_D(B)=\{0\}$ (see \cite[Theorem 1]{Picard1982} or \cite[Section 2]{Pauly2021}). Hence, applying the Theorems \ref{thm:DCLcht} and \ref{thm:locnonloc0} to $a_n=\varepsilon_n$ and $\Omega_{\textnormal{t}}=B$, we are in the position to use condition (ii) in Theorem \ref{thm:DCLcht}. For this, we let $\phi\in C_c^\infty(\Omega)$, $\spt \phi\subseteq B$. Then
\begin{align*}
   \dive \varepsilon_{\kappa(n)} \phi q_n&=    \dive \phi \varepsilon_{\kappa(n)}  q_n \\ &= \langle \grad \phi,\varepsilon_{\kappa(n)}q_n\rangle_{\R^3} + \phi \dive \varepsilon_{\kappa(n)}  q_n \\ &= \langle \grad \phi,\varepsilon_{\kappa(n)}q_n\rangle_{\R^3} + f; 
\end{align*}
thus, $(\dive \varepsilon_{\kappa(n)} \phi q_n)_n$ is bounded in $L_2(B)$.
Also, 
\[
   \curlc \phi q_n = \grad \phi \times q_n,
\]
forming a bounded sequence in $L_2(B)^3.$
Since $B$ particularly satisfies the electric compactness property by the Picard--Weber--Weck selection theorem (Theorem \ref{thm:PWW}), the embeddings $L_2(B)\hookrightarrow H^{-1}(B)$ and $\ran(\curlc,B)\hookrightarrow H^{-1}(\curlr,B)$ are compact. Hence,  by condition (ii) in Theorem \ref{thm:DCLcht}, we obtain, using $q\coloneqq GU$, 
\[
   \varepsilon_{\kappa(n)} \phi q_n \to \varepsilon\phi q =\varepsilon \phi U
\]weakly in $L_2(B)$. As both $B$ and $\phi$ were arbitrary, we infer that
\[
   \varepsilon_{\kappa(n)} q_n \to \varepsilon q =\varepsilon GU
\]weakly in $L_2(\Omega)$. Weak closedness of $G^*$ implies 
\[
   (f,y) = G^*\varepsilon_{\kappa(n)} G U_{\kappa(n)} = G^*\varepsilon_{\kappa(n)} q_n \to G^*\varepsilon G U,
\] which determines $U$ uniquely. Hence, $U_n\to U$ weakly. Repeating the argument replacing $\kappa(n)$ by $n$, we deduce that also $\varepsilon_n G U_n\to \varepsilon G U $ weakly. Since $(f,y)$ are arbitrary, we thus obtain both
\[
  \varepsilon_{n,00}^{-1} \to \varepsilon_{00}^{-1}
\]
and 
\[
   \varepsilon_{n,10} \varepsilon_{n,00}^{-1} \to    \varepsilon_{10} \varepsilon_{00}^{-1}
\] in the respective weak operator topologies. 

Next, we turn to the $\curl$-problem. For this, we let $g\in \ran(\curlc,\Omega)$. Then we find uniquely determined $v_n\in \dom(\curlr)$ such that
\[
    \curlrd \varepsilon_n^{-1}\curlr v_n = g.
\]
Since $\varepsilon_n \in M(\alpha,\beta;\Omega)$, we deduce that $(v_n)_n$ forms a bounded sequence in $\dom(\curlr)$. Thus, there is a weakly convergent subsequence $(r_n)_n\coloneqq (\varepsilon_{\kappa(n)}^{-1}\curlr v_{\kappa(n)})_n$ of $(\varepsilon_{n}^{-1}\curlr v_n)_n$  in $L_2(\Omega)^3$; denote by $r$ its limit. Again, let $B$ be an open ball strictly contained in $\Omega$ and $\phi$ a smooth function supported on $B$. Then appealing to condition (ii) in Theorem \ref{thm:DCLcht} with $q_n = \phi r_n$, we compute
\[
  \dive \varepsilon_{\kappa(n)} \phi r_n = \langle \grad \phi, \varepsilon_{\kappa(n)}  r_n\rangle_{\R^3}
\]
and
\[
  \curlc \phi r_n = \grad \phi \times r_n + g.
\]Hence, both (a) and (b) in the just mentioned condition (ii) are met with the help of Theorem \ref{thm:PWW}. By local $H$-convergence on $B$, we deduce 
\[
  \varepsilon_{\kappa(n)} \phi r_n\to   \varepsilon \phi r.
\] weakly. Thus,
\[
  \curlr v_{\kappa(n)} =  \varepsilon_{\kappa(n)} r_n\to  \varepsilon r
\]weakly. The closedness of $\ran(\curlr)$ yields the existence of some $v\in \dom(\curlr)$ such that $\varepsilon r = \curlr v$. Since, by the (weak) closedness of $\curlc$, 
\[
    g = \curlc r_n \to \curlc r = \curlc \varepsilon^{-1} \curlr v,
\]
$v$ is uniquely determined, which, in turn, uniquely determines $r$. Hence, 
\[
   \varepsilon_{n}^{-1}\curlr v_n \to r = \varepsilon^{-1} \curlr v
\]
weakly. Repeating the argument with $n$ instead of $\kappa(n)$, we also obtain
\[
  \curlr v_n =  \varepsilon_n   \varepsilon_{n}^{-1}\curlr v_n \to \varepsilon r = \varepsilon  \varepsilon^{-1} \curlr v = \curlr v.
\] weakly. Thus, $v_n\to v$ weakly in $\dom(\curlr)$ as $\curlr$ is continuously invertible. Since $g$ was arbitrary, the weak convergence $  \varepsilon_{n}^{-1}\curlr v_n \to   \varepsilon^{-1}\curlr v$ implies $\varepsilon_{n,00}^{-1}\varepsilon_{n,01} \to \varepsilon_{00}^{-1}\varepsilon_{01}$ in the weak operator topology. The weak convergence $v_n\to v$ is the last required convergence in Theorem \ref{thm:charnonlb} (ii).
\end{proof}
We conclude this section stating that the local $H$-convergence topology, $\tau_{\textnormal{H}}$, on $M(\alpha,\beta;\Omega)$ thus coincides with the induced topology of $\tau_{\textnormal{bH}}$ on $M(\alpha,\beta;\Omega)$. 
\begin{corollary}\label{cor:locn} We have
\[
   (M(\alpha,\beta;\Omega),\tau_{\textnormal{H}})=   (M(\alpha,\beta;\Omega),\tau_{\textnormal{bH}}).
\]
\end{corollary}
\begin{proof}
 The fundamental theorem of local $H$-convergence states that $(M(\alpha,\beta;\Omega),\tau_{\textnormal{H}})$ is a compact metric space. $ (M(\alpha,\beta;\Omega),\tau_{\textnormal{bH}})$ is a Hausdorff topological space. By Theorem \ref{thm:locstrongnonloc}, $\id\colon (M(\alpha,\beta;\Omega),\tau_{\textnormal{H}}) \to   (M(\alpha,\beta;\Omega),\tau_{\textnormal{bH}})$ is continuous. Hence, $ (M(\alpha,\beta;\Omega),\tau_{\textnormal{H}})=   (M(\alpha,\beta;\Omega),\tau_{\textnormal{bH}})$.
\end{proof}
 \begin{proof}[Proof of Theorem \ref{thm:locn}] This is Corollary 
 \ref{cor:locn}.
 \end{proof}
\section{Conclusion}\label{sec:concl}

In this article we have introduced a nonlocal $H$-convergence topology for $\Omega$ not necessarily satisfying $\cH_D(\Omega)=\{0\}$. For the case of trivial $\cH_D(\Omega)$, this topological space coincides with the topology of nonlocal $H$-convergence introduced in \cite{W18_NHC}. Moreover, for local coefficients, the newly introduced topology coincides with the topology of local $H$-convergence as introduced by Murat and Tartar in particular if $\dim(\cH_D(\Omega))<\infty$.

We have analysed electrostatic equations with coefficients converging in the newly introduced topology thus obtaining new, nonlocal homogenisation results. These results furthermore led to new compactness results akin to the Picard--Weber--Weck selection theorem with moving coefficients. This type of compactness has potential impact on nonlinear static Maxwell problems, which will be addressed in future publications. Also, the presented notion can be generalised to different boundary conditions as well as other differential operators using the notion of compact Hilbert complexes of linear operators. This, too, will be subject of a future text on this topic.

\section*{Acknowledgements} 

The author wishes to thank Paula Ostmann, Helga Ostmann, Neewa Nopsi and Andreas Buchinger for useful discussions helping to significantly improve the paper. The authors is indebted to the anonymous referee for their thorough and detailed review helping to significantly improve the presentation. The support of the GrK2583/1 during an early stage of this research is gratefully acknowledged.

\bibliographystyle{abbrv}

\end{document}